\documentclass[12pt]{amsart}
\usepackage{fullpage,url,amssymb,enumerate,colonequals}
\usepackage[all]{xy} 
\usepackage{mathrsfs} 

\usepackage[OT2,T1]{fontenc}
\DeclareSymbolFont{cyrletters}{OT2}{wncyr}{m}{n}
\DeclareMathSymbol{\Sha}{\mathalpha}{cyrletters}{"58}

\usepackage{color}

\newcommand{\C}{\mathbb{C}}
\newcommand{\F}{\mathbb{F}}

\newcommand{\PP}{\mathbb{P}}
\newcommand{\Q}{\mathbb{Q}}
\newcommand{\R}{\mathbb{R}}
\newcommand{\Z}{\mathbb{Z}}

\newcommand{\eps}{\varepsilon}

\newcommand{\bmu}{\ensuremath{\boldsymbol\mu}}

\newcommand{\calC}{\mathcal{C}}

\newcommand{\calF}{\mathcal{F}}

\newcommand{\calZ}{\mathcal{Z}}

\DeclareMathOperator{\lcm}{lcm}

\DeclareMathOperator{\ord}{ord}

\DeclareMathOperator{\Res}{Res}

\newcommand{\tors}{{\operatorname{tors}}}

\newcommand{\GL}{\operatorname{GL}}

\newcommand{\PGL}{\operatorname{PGL}}
\newcommand{\PSL}{\operatorname{PSL}}

\newcommand{\pws}[1]{[\![#1]\!]}
\newenvironment{sm}{\left(\begin{smallmatrix}}{\end{smallmatrix}\right)}

\newtheorem{theorem}{Theorem}
\newtheorem{lemma}[theorem]{Lemma}
\newtheorem{corollary}[theorem]{Corollary}
\newtheorem{proposition}[theorem]{Proposition}

\theoremstyle{definition}
\newtheorem{definition}[theorem]{Definition}

\newtheorem{conjecture}[theorem]{Conjecture}

\newtheorem{examples}[theorem]{Examples}

\theoremstyle{remark}

\definecolor{darkgreen}{rgb}{0,0.5,0}
\usepackage[
        colorlinks, citecolor=darkgreen,
        backref,
        pdfauthor={Michael Stoll}, 
        pdftitle={Simultaneous torsion in the Legendre family},
]{hyperref}
\usepackage[alphabetic,backrefs,lite]{amsrefs} 

\usepackage{comment}

\begin{document}

\title[Simultaneous torsion]%
      {Simultaneous torsion in the Legendre family}

\author{Michael Stoll}
\address{Mathematisches Institut,
         Universit\"at Bayreuth,
         95440 Bayreuth, Germany.}
\email{Michael.Stoll@uni-bayreuth.de}
\urladdr{http://www.mathe2.uni-bayreuth.de/stoll/}

\date{October 4, 2015}

\begin{abstract}
  We improve a result due to Masser and Zannier, who showed that the set
  \[ \{\lambda \in \C \setminus \{0,1\} :
         (2,\sqrt{2(2-\lambda)}), (3,\sqrt{6(3-\lambda)}) \in (E_\lambda)_\tors\} \]
  is finite, where $E_\lambda \colon y^2 = x(x-1)(x-\lambda)$ is the Legendre
  family of elliptic curves. More generally, denote by $T(\alpha, \beta)$,
  for $\alpha, \beta \in \C \setminus \{0,1\}$, $\alpha \neq  \beta$, the set
  of $\lambda \in \C \setminus \{0,1\}$ such that all points with $x$-coordinate
  $\alpha$ or~$\beta$ are torsion on~$E_\lambda$. By further results of
  Masser and~Zannier, all these sets are finite. We present a fairly elementary argument
  showing that the set~$T(2,3)$ in question is actually empty. More generally, we obtain
  an explicit description of the set of parameters~$\lambda$ such that
  the points with $x$-coordinate $\alpha$ and $\beta$ are simultaneously torsion,
  in the case that $\alpha$ and~$\beta$ are algebraic numbers that not 2-adically close.

  We also improve another result due to Masser and Zannier dealing with the case
  that $\Q(\alpha, \beta)$ has transcendence degree~$1$. In this case we show
  that $\#T(\alpha, \beta) \le 1$ and that we can decide whether the set is empty
  or not, if we know the irreducible polynomial relating $\alpha$ and~$\beta$.
  This leads to a more precise description of~$T(\alpha, \beta)$ also in the
  case when both $\alpha$ and~$\beta$ are algebraic. We performed extensive
  computations that support several conjectures, for example that there should
  be only finitely many pairs $(\alpha, \beta)$ such that $\#T(\alpha, \beta) \ge 3$.
\end{abstract}

\maketitle


\section{Introduction}

Let
\[ E_\lambda \colon y^2 = x(x-1)(x-\lambda) \]
be the Legendre family of elliptic curves over~$\C$.
For $\alpha \in \C \setminus \{0,1\}$ let $P_\alpha(\lambda)$ be a
point on~$E_\lambda$ with $x$-coordinate~$\alpha$ and set
\[ T(\alpha) = \{\lambda \in \C \setminus \{0,1\}
                   : P_\alpha(\lambda) \in (E_\lambda)_\tors\} \,.
\]
Write $T(\alpha,\beta) = T(\alpha) \cap T(\beta)$.
In~\cites{MasserZannier2008,MasserZannier2010},
Masser and Zannier show that $T(2,3)$ is finite.
This was the first
step in a series of successively more general finiteness results on the set
of parameters such that a given section in a family of two-dimensional
(semi-)abelian varieties is torsion,
see~\cites{MasserZannier2012,MasserZannier2014,MasserZannierPreprint}
(or see the book~\cite{ZannierBook} for an overview).
An alternative, `dynamical' proof of the results
of~\cites{MasserZannier2010,MasserZannier2012} is given by
de~Marco, Wang and Ye in a recent paper~\cite{deMarcoWangYe}.

In this note, we give a $2$-adic proof that $T(2,3)$ is actually empty.
The proof is rather elementary and shows more generally that (for example)
$2$ and~$3$ can be replaced by any pair consisting of an even and an odd
integer (different from $0$ and~$1$). We also give examples of numbers $\alpha$
and~$\beta$ such that $T(\alpha,\beta)$ has exactly one or two elements.
We then give a partial result along the same lines for the two-parameter
Weierstrass family $y^2 = x^3 + Ax + B$.

Returning to the Legendre family, we consider the sets~$T(\alpha, \beta)$
when $\alpha$ and~$\beta$
generate a field of transcendence degree~$1$ over~$\Q$ (the case of
transcendence degree~$2$ is trivial; we have $T(\alpha, \beta) = \emptyset$
in this case). In~\cite{MasserZannier2013}, Masser and Zannier show that,
if we are given an irreducible polynomial $F$ over~$\Q$ such that
$F(\alpha, \beta) = 0$, we can effectively compute the set~$T(\alpha, \beta)$,
and they give a bound on its size:
$\#T(\alpha, \beta) \le 6 \cdot (12 \deg F)^{32}$. We improve this
result considerably; in fact, we prove the best possible bound
$\#T(\alpha, \beta) \le 1$ and also provide better upper bounds for the
occurring torsion orders, leading to a more efficient determination
of the set. We also obtain a fairly precise description of~$T(\alpha, \beta)$
in general. See Proposition~\ref{P:general}. This more precise description
is then used as the basis for extensive computations studying
pairs~$(\alpha, \beta)$ such that $\#T(\alpha, \beta) \ge 2$.
These computations exhibited only a small number of such pairs where
the set~$T(\alpha, \beta)$ has three or more elements, and so we conjecture
that the set of such pairs is actually finite (Conjecture~\ref{conj2}).
Based on our computations, we also conjecture that the heights of
$\alpha$ and~$\beta$ are uniformly bounded when $\#T(\alpha, \beta) \ge 2$
(Conjecture~\ref{conj4}).

This note is organized as follows. We first prove a general statement on
the $2$-adic behavior of elements in a ring defined by a certain kind of
recurrence relation. We then apply this to the division polynomials of the
Legendre elliptic curve. This allows us to deduce `mod~$2$' information on
the set~$T(\alpha)$, for $\alpha \in \bar{\Q}$, leading to our first main result
that $T(\alpha,\beta) \subseteq \{\alpha, \beta\}$ if $\alpha$ and~$\beta$
are distinct `mod~$2$' (see Corollary~\ref{cor1})
or even `mod~$4$' (Corollary~\ref{C:rat} for rational $\alpha$, $\beta$).
We use this to show that the intersection of~$T(\alpha)$ with the set
of all roots of unity can be determined effectively; the set has size
at most~3, and we determine all~$\alpha$ that reach this bound.
We also apply our approach to
the Weierstrass family $y^2 = x^3 + Ax + B$. This leads to a partial
result for the set of parameters~$(A,B)$ such that three $x$-coordinates
are simultaneously torsion. The restriction is that we need to assume that $B$
is integral at~$2$. See Corollary~\ref{C:W}.
We then turn to the case of transcendence degree~$1$ in the Legendre family
and prove our second main result.
The description of~$T(\alpha, \beta)$ obtained as a consequence
of this result is then used as the basis for the computations mentioned
above. We report on the results and state the conjectures already mentioned.

\subsection*{Acknowledgments}

I would like to thank the organizers of the \emph{Second ERC Research Period
on Diophantine Geometry} for inviting me to attend this event; the first result
presented here was obtained during the meeting. I would also like to thank
David Masser and Umberto Zannier for fruitful discussions. The computations
we report on in Sections \ref{S:trdeg1} and~\ref{S:size} were performed
with the computer algebra system Magma~\cite{Magma}.


\section{2-adic behavior of division polynomials}

Let $R$ be a commutative ring and fix elements $f, g \in R$.
Let $(h_n)_{n \ge 1}$ be a sequence of elements of~$R$ satisfying
\[ h_1 = 1, \quad h_2 = 1, \quad h_3 \equiv -g^2 \bmod 4 R \quad\text{and}\quad
   h_4 \equiv 2 g^3 \bmod 4 R
\]
and the recurrence relations (for $m \ge 3, 1, 2$, respectively)
\begin{align*}
  h_{2m} &= h_m \bigl(h_{m+2} h_{m-1}^2 - h_{m-2} h_{m+1}^2\bigr) \\
  h_{4m+1} &= 4 f h_{2m+2} h_{2m}^3 - h_{2m-1} h_{2m+1}^3 \\
  h_{4m-1} &= h_{2m+1} h_{2m-1}^3 - 4 f h_{2m-2} h_{2m}^3 \,.
\end{align*}
(Relations of this form are satisfied by the division polynomials of an
elliptic curve; we will apply the results of this section soon in this setting.)

We define, for $n \in \Z_{>0}$,
\[ d(n) = \left\lfloor\frac{n^2-1}{4}\right\rfloor \]
and
\[ e(n) = \max\{0, v_2(n)-1\} \,, \]
where $v_2$ denotes the $2$-adic valuation.

\begin{proposition} \label{P:rec}
  For $n \in \Z_{>0}$, we have
  \[ h_n \equiv 2^{e(n)} g^{d(n)} \bmod 2^{e(n)+1} R \,. \]
\end{proposition}

\begin{proof}
  We first determine $h_n$ mod~$4 R$:
  We have
  \begin{align*}
    h_{2m+1} &\equiv (-1)^m g^{d(2m+1)} \bmod 4 R \\
    h_{4m+2} &\equiv (-1)^m g^{d(4m+2)} \bmod 4 R \\
    h_{8m+4} &\equiv 2 g^{d(8m+4)} \bmod 4 R \\
    h_{8m}   &\equiv 0 \bmod 4 R
  \end{align*}

  The statements are correct by assumption for $h_n$ with $n \le 4$.
  We proceed by induction using the recurrence relations.
  All congruences below are mod~$4 R$.
  \begin{align*}
    h_{4m+1} &\equiv -h_{2m-1} h_{2m+1}^3 \\
                &\equiv -(-1)^{m-1} (-1)^{3m} g^{d(2m-1)+3d(2m+1)}
                  = (-1)^{2m} g^{d(4m+1)} \\
    h_{4m-1} &\equiv h_{2m+1} h_{2m-1}^3 \\
                &\equiv (-1)^m (-1)^{3(m-1)} g^{d(2m+1)+3d(2m-1)}
                  = (-1)^{2m-1} g^{d(4m-1)} \\
    h_{8m+2} &= h_{4m+1} (h_{4m+3} h_{4m}^2 - h_{4m-1} h_{4m+2}^2) \\
                &\equiv (-1)^{2m}(-(-1)^{2m-1})
                        g^{d(4m+1)+d(4m-1)+2d(4m+2)}
                  = (-1)^{2m} g^{d(8m+2)} \\
    h_{8m-2} &= h_{4m-1} (h_{4m+1} h_{4m-2}^2 - h_{4m-3} h_{4m}^2) \\
                &\equiv (-1)^{2m-1} (-1)^{2m}
                        g^{d(4m-1)+d(4m+1)+2d(4m-2)}
                  = (-1)^{2m-1} g^{d(8m-2)} \\
    h_{16m+4} &= h_{8m+2} (h_{8m+4} h_{8m+1}^2 - h_{8m} h_{8m+3}^2) \\
                  &\equiv 2 g^{d(8m+2)+d(8m+4)+2d(8m+1)}
                  = 2 g^{d(16m+4)} \\
    h_{16m-4} &= h_{8m-2} (h_{8m} h_{8m-3}^2 - h_{8m-4} h_{8m-1}^2) \\
                  &\equiv 2 g^{d(8m-2)+d(8m-4)+2d(8m-1)}
                  = 2 g^{d(16m-4)} \\
    h_{16m+8} &= h_{8m+4} (h_{8m+6} h_{8m+3}^2 - h_{8m+2} h_{8m+5}^2) \\
                  &\equiv 2\bigl(g^{d(8m+4)+d(8m+6)+2d(8m+3)}
                                - g^{d(8m+4)+d(8m+2)+2d(8m+5)}\bigr)
                  \equiv 0 \\
    h_{16m} &= h_{8m} (h_{8m+2} h_{8m-1}^2 - h_{8m-2} h_{8m+1}^2)
                \equiv 0
  \end{align*}
  The relations $d(2m-1)+3d(2m+1) = d(4m+1)$ etc.\ are easily verified.

  This shows the claim when $e(n) \le 1$. We now show by induction on~$e(n)$
  that it holds in general. So let $n = 2^{e+1} m$ with $e \ge 2$ and $m$~odd. Then
  \[ h_n = h_{2^e m} (h_{2^e m+2} h_{2^e m-1}^2 - h_{2^e m-2} h_{2^e m+1}^2) \,. \]
  The second factor is (mod~$4 R$)
  \begin{align*}
    h_{2^e m+2} h_{2^e m-1}^2 - h_{2^e m-2} h_{2^e m+1}^2
      &\equiv (-1)^{2^{e-2} m} g^{d(2^e m+1) + 2d(2^e m-1)} \\
      & \qquad\qquad {} - (-1)^{2^{e-2} m-1} g^{d(2^e m-2) + 2d(2^e m+1)} \\
      &\equiv 2 g^{d(2^{e+1} m) - d(2^e m)} \,,
  \end{align*}
  whereas the first is
  \[ h_{2^e m} \equiv 2^{e-1} g^{d(2^e m)} \bmod 2^e R \,. \]
  Multiplying gives the desired congruence
  \[ h_{2^{e+1} m} \equiv 2^e g^{d(2^{e+1} m)} \bmod 2^{e+1} R \,. \qedhere \]
\end{proof}


\section{Application to the Legendre family}

We consider the Legendre family $E_\lambda \colon y^2 = x(x-1)(x-\lambda)$
of elliptic curves.
We denote by $\psi_n(\lambda,x)$ the $n$th reduced division polynomial of~$E_\lambda$;
its roots are the $x$-coordinates of the points of order dividing~$n$ and $> 2$.
These polynomials are related to the `bicyclotomic polynomials'~$B^*_n(x,T)$
of Masser and Zannier~\cite{MasserZannier2013} via
\[ \psi_n(\lambda, x) = \prod_{2 \neq d \mid n} B^*_d(x, \lambda) \,. \]
We have $\psi_1 = \psi_2 = 1$,
\begin{align*}
  \psi_3(\lambda, x) &= 3 x^4 - 4(\lambda + 1) x^3 + 6 \lambda x^2 - \lambda^2 \\
  \psi_4(\lambda, x) &= 2 (x^2 - \lambda) (x^2 - 2 x + \lambda) (x^2 - 2 \lambda x + \lambda)
\end{align*}
and
\begin{align*}
  \psi_{2m}(\lambda, x)
     &= \psi_m(\lambda, x) \bigl(\psi_{m+2}(\lambda, x) \psi_{m-1}(\lambda, x)^2
                                  - \psi_{m-2}(\lambda, x) \psi_{m+1}(\lambda, x)^2\bigr) \\
  \psi_{4m+1}(\lambda, x)
     &= 16 x^2 (x-1)^2 (x-\lambda)^2 \psi_{2m+2}(\lambda, x) \psi_{2m}(\lambda, x)^3
          - \psi_{2m-1}(\lambda, x) \psi_{2m+1}(\lambda, x)^3 \\
  \psi_{4m-1}(\lambda, x)
     &= \psi_{2m+1}(\lambda, x) \psi_{2m-1}(\lambda, x)^3
         - 16 x^2 (x-1)^2 (x-\lambda)^2 \psi_{2m-2}(\lambda, x) \psi_{2m}(\lambda, x)^3
\end{align*}
It follows that $\psi_n(\lambda, x) \in \Z[\lambda, x]$ for all $n \ge 1$.

We note that
\begin{align*}
  \psi_3(\lambda, x) &\equiv -(\lambda - x^2)^2 \bmod 4 \Z[\lambda, x] \qquad\text{and} \\
  \psi_4(\lambda, x) &\equiv 2 (\lambda - x^2)^3 \bmod 4 \Z[\lambda, x] \,.
\end{align*}
So we can apply Proposition~\ref{P:rec} with $R = \Z[\lambda, x]$,
$f = 4 x^2 (x-1)^2 (x-\lambda)^2$ and $g = \lambda - x^2$. This gives the following.

\begin{proposition} \label{P:psi}
  For $n \in \Z_{>0}$, we have
  \[ \psi_n(\lambda, x) \equiv 2^{e(n)} (\lambda - x^2)^{d(n)}
                               \bmod 2^{e(n)+1} \Z[\lambda, x] \,.
  \]
  Furthermore, $\deg_{\lambda} \psi_n(\lambda, x) = d(n)$ and
  $\deg \psi_n(\lambda, x) = \deg_x \psi_n(\lambda, x) = 2 d(n)$,
  where $\deg$ denotes the total degree.
\end{proposition}

\begin{proof}
  The congruence follows from Proposition~\ref{P:rec}.
  The upper bounds $\deg_{\lambda} \psi_n(\lambda, x) \le d(n)$
  and $\deg_x \psi_n(\lambda, x) \le \deg \psi_n(\lambda, x) \le 2 d(n)$
  follow easily by induction, using the recurrence relations.
  Since the reduction of $\psi_n(\lambda, x)$ modulo a suitable power of~$2$
  has $\lambda$-degree~$d(n)$ and $x$-degree~$2 d(n)$,
  we actually have equality.
\end{proof}

Recall the definitions
\[ T(\alpha) = \{\lambda \in \C \setminus \{0,1\} :
                        P_\alpha(\lambda) \in (E_\lambda)_\tors\}
   \quad\text{and}\quad T(\alpha,\beta) = T(\alpha) \cap T(\beta)\,.
\]
It is clear that $T(\alpha) \subseteq \bar{\Q}$ if $\alpha \in \bar{\Q}$.
More generally, if $\lambda \in T(\alpha)$, then $\lambda$ is algebraic
over~$\Q(\alpha)$ and $\alpha$ is algebraic over~$\Q(\lambda)$.
This immediately implies that $T(\alpha,\beta) = \emptyset$
whenever the transcendence degree of~$\Q(\alpha,\beta)$ is~$2$
(compare~\cite{MasserZannier2013}*{p.~636}).
We will now consider the other extreme, when $\alpha$ and~$\beta$
are both algebraic over~$\Q$.

For the following, fix an embedding $i \colon \bar{\Q} \hookrightarrow \bar{\Q}_2$.
Write $Z \subseteq \bar{\Q}$ for the subring of elements~$\alpha$ such that $i(\alpha)$
is integral and denote the natural `reduction' map
$\bar{\Q} \hookrightarrow \PP^1(\bar{\Q}_2) \to \PP^1(\bar{\F}_2)$
by~$\rho$. We write $v \colon \bar{\Q} \to \Q \cup \{\infty\}$ for the valuation
associated to~$i$, normalized such that $v(2) = 1$.

\begin{theorem} \label{T:Talpha}
  Let $\alpha \in \bar{\Q} \setminus \{0,1\}$.
  Then $T(\alpha) = \{\alpha\} \cup T'(\alpha)$,
  where $\rho\bigl(T'(\alpha)\bigr) \subseteq \{\rho(\alpha^2)\}$.
\end{theorem}

\begin{proof}
  Let $\lambda \in T(\alpha)$ and
  let $n \ge 2$ be the order of the point $P_\alpha(\lambda) \in E_\lambda(\bar{\Q})$.
  If $n = 2$, then $\lambda = \alpha$. Otherwise $n \ge 3$ and
  $\psi_n(\lambda,\alpha) = 0$.

  First assume $\rho(\alpha) \neq \infty$, so that $\alpha \in Z$.
  Proposition~\ref{P:psi} then shows that $2^{-e(n)} \psi_n(t, \alpha) \in Z[t]$ with
  unit leading coefficient, so $\lambda \in Z$ and
  \[ 0 = 2^{-e(n)} \psi_n(\lambda, \alpha) \equiv (\lambda - \alpha^2)^{d(n)} \bmod 2 Z \,, \]
  which implies $\rho(\lambda) = \rho(\alpha^2)$ (note that $d(n) > 0$ for $n \ge 3$).

  Now consider the case $\rho(\alpha) = \infty$.
  Assuming that $\lambda \in Z$, Proposition~\ref{P:psi} shows that the term coming from the
  monomial~$x^{2d(n)}$ is the unique term in~$2^{-e(n)} \psi_n(\lambda,\alpha)$ with
  minimal valuation ($ = 2d(n) v(\alpha)$), so $\psi_n(\lambda,\alpha)$ cannot vanish.
  This shows that $\lambda \in \bar{\Q} \setminus Z$,
  so $\rho(\lambda) = \infty = \rho(\alpha^2)$.
\end{proof}

We note that Mavraki~\cite{MavrakiPreprint}*{Section~4} has recently given an alternative proof
based on the $2$-adic dynamics of the associated Latt\`es map. We come back to this
approach after stating the following easy corollary.

\begin{corollary} \label{cor1}
  Let $\alpha, \beta \in \bar{\Q} \setminus \{0,1\}$
  such that $\rho(\alpha) \neq \rho(\beta)$. Then
  \[ T(\alpha, \beta) \subseteq \{\alpha, \beta\} \,. \]
  In particular, $T(\alpha, \beta)$ is finite and effectively computable.
\end{corollary}

\begin{proof}
  Theorem~\ref{T:Talpha} shows that any $\lambda \in T(\alpha, \beta) \setminus \{\alpha,\beta\}$
  must satisfy $\rho(\lambda) = \rho(\alpha^2) = \rho(\beta^2)$.
  The existence of such a $\lambda$ would imply that $\rho(\alpha) = \rho(\beta)$
  (recall that squaring is a bijection on~$\PP^1(\bar{\F}_2)$),
  contradicting the assumption. Regarding the effectivity statement, note that
  it is easy to decide for any given~$\lambda$ if $\lambda \in T(\alpha, \beta)$ or
  not: just check if the points with $x$-coordinate $\alpha$ or~$\beta$ are torsion
  on~$E_\lambda$.
\end{proof}

To get somewhat stronger results, we use the Latt\`es map
\[ f_\lambda \colon x \longmapsto \frac{(x^2 - \lambda)^2}{4 x (x-1) (x-\lambda)} \]
that expresses the $x$-coordinate of~$2P$ in terms of the $x$-coordinate of~$P$,
for a point $P \in E_\lambda$. Then $T(\alpha)$ can also be characterized as the
set of $\lambda \in \C \setminus \{0,1\}$ such that $\alpha$ is preperiodic under
iteration of~$f_\lambda$ on~$\PP^1$. We will use the obvious fact that $P$ is torsion
if and only if $2P$ is, which implies that
\[ \lambda \in T(\alpha) \iff
   f_\lambda(\alpha) \in \{0,1,\lambda,\infty\} \quad\text{or}\quad \lambda \in T(f_\lambda(\alpha)) \,.
\]
Note that (for $\alpha, \lambda \neq 0, 1$)
\[ f_\lambda(\alpha) \in \{0,1,\lambda,\infty\} \iff
     \lambda \in \Bigl\{\alpha, \alpha^2, \alpha(2-\alpha), \frac{\alpha^2}{2\alpha-1}\Bigr\} \,.
\]

\begin{lemma} \label{L:Talpha2}
  Let $\alpha, \lambda \in \bar{\Q} \setminus \{0,1\}$ such that $\lambda \in T(\alpha)$.
  Then $\lambda \in \{\alpha, \alpha^2, \alpha(2-\alpha), \alpha^2/(2\alpha-1)\}$
  or $\rho(f_\lambda(\alpha)) = \rho(\alpha)$.
\end{lemma}

\begin{proof}
  If $P_\alpha(\lambda)$ has order dividing~4, then $\lambda$ is in the first set.
  Otherwise $2P_\alpha(\lambda) = P_{f_\lambda(\alpha)}(\lambda)$ is a point of order $> 2$,
  and the claim follows from (the proof of) Theorem~\ref{T:Talpha}, which tells us
  that $\rho(f_\lambda(\alpha)^2) = \rho(\lambda) = \rho(\alpha^2)$.
\end{proof}

We use this to strengthen Theorem~\ref{T:Talpha} in the following way.

\begin{theorem} \label{T:Talpha2}
  Let $\alpha \in Z \setminus \{0, 1\}$. Then
  \begin{align*}
    T(\alpha) &\subseteq \Bigl\{\alpha, \alpha^2, \alpha (2-\alpha), \frac{\alpha^2}{2\alpha-1}\Bigr\}
        \cup \{\alpha^2 + 2 u \alpha (1-\alpha) : u \in Z, \rho(u^2) = \rho(\alpha)\} \\
              &\subseteq \{\alpha\} \cup (\alpha^2 + 2 Z) \,.
  \end{align*}
  For $\alpha \in \bar{\Q}$ with $\rho(\alpha) = \infty$ we have
  \[ T(\alpha) \subseteq \Bigl\{\alpha, \alpha^2, \alpha (2-\alpha), \frac{\alpha^2}{2\alpha-1}\Bigr\}
        \cup \Bigl\{\frac{\alpha^2}{1 + 2(\alpha-1) u} : u \in Z, \rho(u) = 0\Bigr\} \,.
  \]
\end{theorem}

\begin{proof}
  We use Lemma~\ref{L:Talpha2}; we have to show that $\rho(f_\lambda(\alpha)) = \rho(\alpha)$
  implies that $\lambda$ is in the secon set in the union in each case.

  We first assume $\rho(\alpha) \notin \{0, 1, \infty\}$.
  Then $\rho(\lambda) = \rho(\alpha^2) \neq \rho(\alpha)$,
  so $\alpha (\alpha-1) (\alpha-\lambda) \in Z^\times$, and a necessary condition is
  that $f_\lambda(\alpha) \in Z^\times$, which is equivalent to $2 v(\alpha^2 - \lambda) = 2$,
  so $\lambda = \alpha^2 + 2 u \alpha (1-\alpha)$ with some $u \in Z^\times$.

  Next we consider $\rho(\alpha) = 1$. Write $\alpha = 1 + \delta$ and $\lambda = 1 + \delta - \eps$
  with $v(\delta), v(\eps) > 0$. We have
  $\alpha (\alpha - 1) (\alpha - \lambda) = \alpha \delta \eps$
  and $\alpha^2 - \lambda = \delta(\delta + 1) + \eps$, so the necessary
  condition $v(f_\lambda(\alpha)) = 0$ means
  \[ 2 v(\delta(\delta+1) + \eps) = 2 + v(\delta) + v(\eps) \,. \]
  If $v(\delta) \neq v(\eps)$, then we obtain the contradiction
  \[ 2 + v(\delta) + v(\eps) > 2 + 2 \min\{v(\delta), v(\eps)\} > 2 v(\delta(\delta+1) + \eps) \,. \]
  Otherwise, we find that $v(\alpha^2 - \lambda) = 1 + v(\delta) = v(2\delta)$,
  so again $\lambda = \alpha^2 + 2 \alpha (1 - \alpha) u$ with a unit $u \in Z^\times$.
  Using this in the expression for~$f_\lambda(\alpha)$, we find (in both cases considered) that
  \[ \rho(\alpha) \stackrel{!}{=} \rho(f_\lambda(\alpha)) = \rho(u^2) \,. \]

  The cases $\rho(\alpha) = 0$ and $\rho(\alpha) = \infty$ can be reduced to $\rho(\alpha) = 1$
  by noting that $\lambda \in T(\alpha)$ is equivalent to $1-\lambda \in T(1-\alpha)$
  and to $1/\lambda \in T(1/\alpha)$.
\end{proof}

Note that we have only used one step in the iteration of~$f_\lambda$, so further improvements 
should be possible.

When $\alpha \in \bar{\Q} \setminus \{0,1\}$ we write
\[ R(\alpha) = \begin{cases}
                 \{\alpha^2 + 2 u \alpha (1-\alpha) : u \in Z, \rho(u^2) = \rho(\alpha)\}
                    & \text{if $\alpha \in Z$,} \\[2mm]
                 \Bigl\{\frac{\alpha^2}{1 + 2(\alpha-1) u} : u \in Z, \rho(u) = 0\Bigr\}
                    & \text{otherwise}
               \end{cases}
\]
and
\[ S(\alpha) = \Bigl\{\alpha, \alpha^2, \alpha (2-\alpha), \frac{\alpha^2}{2\alpha-1}\Bigr\} \,. \]

\begin{proposition} \label{P:finite}
  Let $\alpha, \beta \in \bar{\Q} \setminus \{0, 1\}$.
  If $R(\alpha) \cap R(\beta) = \emptyset$, then
  \[ T(\alpha, \beta) \subseteq \bigl(S(\alpha) \cap S(\beta)\bigr)
                                 \cup \bigl(S(\alpha) \cap R(\beta)\bigr)
                                 \cup \bigl(R(\alpha) \cap S(\beta)\bigr)
                      \subseteq S(\alpha) \cup S(\beta) \,.
  \]
  In particular, $T(\alpha, \beta)$ is finite and effectively computable.

  The condition $R(\alpha) \cap R(\beta) = \emptyset$ holds in the following
  situations.
  \begin{enumerate}[\upshape(1)]\addtolength{\itemsep}{1mm}
    \item $\rho(\alpha) \neq \rho(\beta)$;
    \item $\rho(\alpha) = \rho(\beta) \notin \{0,1,\infty\}$ and $v(\alpha-\beta) \le 1/2$;
    \item $\rho(\alpha) = \rho(\beta) = 1$, $0 < v(\alpha-1) \le 1$ and
          $v(\alpha - \beta) = v(\alpha - 1)$;
    \item $\rho(\alpha) = \rho(\beta) = 0$, $v(\alpha) \le 1$ and $v(\alpha - \beta) = v(\alpha)$;
    \item $\rho(\alpha) = \rho(\beta) = \infty$, $v(\alpha) \ge -1$ and $v(\alpha - \beta) = v(\beta)$.
  \end{enumerate}
\end{proposition}

\begin{proof}
  The first statement is clear, since by Theorem~\ref{T:Talpha2},
  $T(\alpha) \subseteq S(\alpha) \cup R(\alpha)$ and $S(\alpha)$ is finite.

  Case~(1) was already dealt with in Corollary~\ref{cor1}.
  For case~(2), we observe that $v(\alpha - \beta) \le 1/2$ implies $v(\alpha^2 - \beta^2) \le 1$.
  The difference~$\delta$ of an element in~$R(\alpha)$ and an element of~$R(\beta)$ satisfies
  $v(\delta - (\alpha^2 - \beta^2)) > 1$, which implies that $\delta$ cannot be zero.
  In case~(3), we write $\alpha = 1 + \eps$, $\beta = 1 + \eps'$; then $v(\eps) \le 1$
  and either $v(\eps') > v(\eps)$ or $v(\eps') = v(\eps) = v(\eps-\eps')$.
  An element of~$R(\alpha)$ has the form
  $1 + 2 \eps + \eps^2 + 2 \eps (1 + \eps) (1 + \eta_1) = 1 + \eps^2 + 2 \eps \eta$
  where $v(\eta_1), v(\eta) > 0$, and similarly for~$R(\beta)$. So the difference
  is $\delta = \eps^2 - {\eps'}^2 + 2 (\eps \eta - \eps' \eta')$.
  If $v(\eps) < v(\eps')$, then $v(\delta) = 2 v(\eps)$, so $\delta \neq 0$.
  In the other case $v(\delta) = v(\eps^2 - {\eps'}^2) = 2 v(\eps - \eps') = 2v(\eps)$,
  so again $\delta \neq 0$. The remaining cases can be reduced to case~(3) in the usual way.
\end{proof}

We consider the case of rational numbers in more detail.

\begin{corollary} \label{C:rat}
  Let $\alpha, \beta \in \Q \setminus \{0, 1\}$.
  \begin{enumerate}[\upshape(1)]\addtolength{\itemsep}{1mm}
    \item If $\rho(\alpha) \neq \rho(\beta)$, then $T(\alpha, \beta) = \emptyset$.
    \item If $\alpha \equiv 3 \bmod 4$ and $\beta \equiv 1 \bmod 4$,
          then $T(\alpha, \beta) \subseteq \{\alpha^2, \beta\}$.
    \item If $\alpha \equiv 2 \bmod 4$ and $\beta \equiv 0 \bmod 4$,
          then $T(\alpha, \beta) \subseteq \{\alpha (2-\alpha), \beta\}$.
    \item If $v(\alpha) = -1$ and $v(\beta) \le -2$, then
          $T(\alpha, \beta) \subseteq \{\alpha^2/(2\alpha-1), \beta\}$.
  \end{enumerate}
\end{corollary}

\begin{proof}
  For (1) see Corollary~\ref{cor1} and note that $\rho(\alpha^2) = \rho(\alpha)$
  and $\rho(\beta^2) = \rho(\beta)$.
  Statement~(2) follows by obseerving that
  all elements of~$T(\beta)$ except possibly~$\beta$ are in $1 + 8Z$, whereas
  $\alpha^2$ is the only element of~$T(\alpha)$ with this property.
  Parts (3) and~(4) are deduced from~(2).
\end{proof}

This can be interpreted as saying that when $\alpha$ and~$\beta$ `differ mod~4',
then we can determine $T(\alpha, \beta)$ effectively and the set has at most
two elements.

\begin{examples}
  We apply the results above to give examples of numbers
  $\alpha, \beta \in \bar{\Q} \setminus \{0,1\}$ such that
  $T(\alpha, \beta)$ can be determined explicitly and has zero, one or two elements.
  \begin{enumerate}[(1)]\addtolength{\itemsep}{1mm}
    \item $T(2, 3) = \emptyset$. This is a special case of Corollary~\ref{C:rat}~(1).
    \item Let $\omega$ be a primitive cube root of unity. Then
          $T(\omega, \omega^2) = \{\omega, \omega^2\}$.
          The second statement of Corollary~\ref{cor1} gives the inclusion `$\subseteq$'.
          It is easily checked that $P_\omega(\omega)$ and $P_{\omega^2}(\omega^2)$
          have order~$2$, while $P_\omega(\omega^2)$ and $P_{\omega^2}(\omega)$
          have order~$4$.
    \item $T(2,4) = \{4\}$. The inclusion $T(2,4) \subseteq \{4\}$ follows from
          Corollary~\ref{C:rat}~(3) (recall that zero is not a permissible value).
          On the other hand, $4 = 2^2$ is clearly in~$T(2,4)$.
    \item $T(3,-3) = \{-3,9\}$. The inclusion `$\subseteq$' follows from
          Corollary~\ref{C:rat}~(2). Clearly $9 = 3^2 = (-3)^2 \in T(3,-3)$,
          and one checks that $-3 \in T(3)$.
  \end{enumerate}
\end{examples}

In a similar (but even simpler) way as we did it above regarding the
$2$-adic behavior of the~$\psi_n$, one can show the following.

\begin{proposition} \label{P:xfixed}
  For every $n \ge 1$, we have
  \begin{align*}
    \psi_n(\lambda, 0) &= a_n \lambda^{d(n)} \\
    \psi_n(\lambda, 1) &= a_n (1 - \lambda)^{d(n)} \\
    \psi_n(\lambda, \lambda) &= a_n \bigl(\lambda (1-\lambda)\bigr)^{d(n)}
  \end{align*}
  where $a_{2m+1} = (-1)^m$ and $a_{2m} = (-1)^{m-1} m$.
\end{proposition}

{}From this, one can conclude that if $a$ and~$b$ are integers and $p$ is
a prime such that $a \equiv 0 \bmod p$ and $b \equiv 1 \bmod p$,
then for any $\lambda \in T(a,b) \setminus \{a,b\}$,
the order of the points $P_a(\lambda)$ and~$P_b(\lambda)$ must be
a multiple of~$2p$. Since this result is much weaker than what can be
obtained from the consideration of the $2$-adic behavior, we will not
pursue this further here. It may be worth while, however, to study
the $p$-adic behavior of the polynomials~$\psi_n$ for $p \neq 2$
in some detail.


\section{An unlikely intersection problem of Habegger, Jones and Masser}

In a recent preprint~\cite{HabeggerJonesMasser} Habegger, Jones and Masser
consider various specific unlikely intersection problems, one of which
asks for the set $T(2) \cap \bmu$, where $\bmu = \exp(2\pi i \Q) \subseteq \C$
denotes the set of roots of unity. The result they obtain in this case
(Theorem~5 in loc.~cit.) is that there is an effective constant $C > 0$
such that $[\Q(\zeta) : \Q] \le C$ for every $\zeta \in T(2) \cap \bmu$.
In this section we use the results of the previous section to obtain a much stronger result.

We continue to work with the ring $Z \subseteq \bar{\Q}$, the reduction map
$\rho \colon \bar{\Q} \to \PP^1(\bar{\F}_2)$ and the valuation map
$v \colon \bar{\Q} \to \Q \cup \{\infty\}$.

\begin{corollary}
  Let $\alpha \in \bar{\Q}$ be such that some conjugate of~$\alpha$
  is not in~$Z^\times$ (i.e., is not a 2-adic unit).
  Then $T(\alpha) \cap \bmu = \emptyset$.
  In particular, $T(2) \cap \bmu = \emptyset$.
\end{corollary}

\begin{proof}
  After applying an automorphism of~$\bar{\Q}$ we can assume
  that $\alpha \notin Z^\times$, which implies that $\rho(\alpha) \in \{0, \infty\}$.
  By Theorem~\ref{T:Talpha} we have $T(\alpha) = \{\alpha\} \cup T'(\alpha)$
  with $\rho(T'(\alpha)) \subseteq \{\rho(\alpha^2)\}$. So
  $\rho(T(\alpha)) \subseteq \{0\}$ or $\rho(T(\alpha)) \subseteq \{\infty\}$.
  Since clearly
  $\rho(\bmu) \cap \{0, \infty\} = \emptyset$, the claim follows.
\end{proof}

The case of 2-adic units is more interesting.
Note that $\alpha \in T(\alpha) \cap \bmu$ when $\alpha \in \bmu$,
so we can definitely have non-empty intersections in this case.
Theorem~\ref{T:Talpha} tells us that any $\alpha \neq \zeta \in T(\alpha) \cap \bmu$
must satisfy $\rho(\zeta) = \rho(\alpha^2)$. There is a unique $\zeta_0 \in \bmu$
of odd order satisfying this requirement, and we obtain that
\[ T'(\alpha) \cap \bmu \subseteq \{\zeta_0 \zeta : \zeta \in \bmu_{2^\infty}\} \]
where $\bmu_{2^\infty}$ denotes the group of roots of unity of order~$2^m$
for some~$m$. From Theorem~\ref{T:Talpha2} we get the more precise requirement
\[ T'(\alpha) \subseteq \alpha^2 + 2 Z \,. \]
Write $\zeta_0^{-1} \alpha^2 = 1 + \eps$ with $v(\eps) > 0$. Then we must have
\[ \zeta \equiv 1 + \eps \bmod 2 Z \,. \]
This leads to the following.

\begin{corollary}
  Let $\alpha \in \bar{\Q}$ be such that all conjugates of~$\alpha$
  are in~$Z^\times$. Then $T(\alpha) \cap \bmu$ has at most two elements
  different from~$\alpha$; the set can be effectively determined.

  If $\alpha \in \bmu$, then
  \[ \{\alpha, \alpha^2\} \subseteq T(\alpha) \cap \bmu \subseteq \{\alpha, \alpha^2, -\alpha^2\} \]
  except when $\alpha = -1$, where we have $T(-1) \cap \bmu = \{-1\}$.
\end{corollary}

\begin{proof}
  We note first that $\bmu \cap (1 + 2 Z) = \{-1,1\}$.
  Assume that $\zeta \in \bmu$
  satisfies $\zeta \equiv 1 + \eps \bmod 2$. If $\zeta'$ is another such
  root of unity, then $\zeta' \zeta^{-1} \in 1 + 2Z$ and so $\zeta' = \pm \zeta$.
  We conclude that $T'(\alpha) \cap \bmu \subseteq \{\pm \zeta_0 \zeta\}$.
  Note that we can effectively decide whether $\zeta$ exists, and if so, find it.
  This implies effectivity.

  For the second statement note that $\alpha$ and~$\alpha^2$ (unless $\alpha^2 = 1$)
  are always in~$T(\alpha)$. If $\alpha \in \bmu$, then we can take $\zeta_0 = \alpha^2$
  and $\eps = 0$ in the argument above, so that $-\zeta_0 = -\alpha^2$ is the
  only remaining possibility. When $\alpha = -1$, we have
  $\{\alpha, \pm\alpha^2\} \setminus \{0,1\} = \{-1\}$.
\end{proof}

When $\alpha \in \bmu$, we can actually rule out the occurrence of~$-\alpha^2$
in most cases. Note that $-\alpha^2 \in T(\alpha)$ implies that
\[ \rho(\alpha) = \rho(f_{-\alpha^2}(\alpha)) = \rho\Bigl(\frac{\alpha^2}{\alpha^2-1}\Bigr) \,. \]
This implies $\rho(\alpha^2 + \alpha + 1) = 0$, which
means that the order of~$\alpha$ is of the form $3 \cdot 2^m$.
In the next step, we have
\[ f_{-\alpha^2}\Bigl(\frac{\alpha^2}{\alpha^2-1}\Bigr)
                = \frac{(\alpha^4 - \alpha^2 + 1)^2}{4 \alpha^2 (\alpha^2 - 1)} \,,
\]
which must be zero or a 2-adic unit. In the first case
$\alpha^4 - \alpha^2 + 1 = 0$, which is satisfied by
the primitive 12th roots of unity. Otherwise we must have $v(\alpha^4 - \alpha^2 + 1) = 1$.
Since $\alpha^4 - \alpha^2 + 1 \equiv (\alpha^2 + \alpha + 1)^2 \bmod 2 Z$
and $\alpha - 1$ is a unit, we get $\alpha^3 = 1 + \eps$ with $v(\eps) \ge 1/2$,
and we know that $(1 + \eps)^{2^m} = 1$ for some~$m$.
This implies $m \le 2$, so that the order of~$\alpha$ is 3, 6 or~12.
One can check that in each of these cases we have indeed $-\alpha^2 \in T(\alpha)$.
We summarize our findings.

\begin{proposition}
  If $\alpha \in \bmu$, then
  \[ T(\alpha) \cap \bmu =
        \begin{cases}
          \{\alpha\} & \text{if $\alpha = -1$,} \\
          \{\alpha, \alpha^2, -\alpha^2\} & \text{if $\ord(\alpha) \in \{3,6,12\}$,} \\
          \{\alpha, \alpha^2\} & \text{otherwise.}
        \end{cases}
   \]
\end{proposition}

Together with the previous results of this section, this implies that
\[ \max_{\alpha \in \C \setminus \{0,1\}} \#\bigl(T(\alpha) \cap \bmu\bigr) = 3 \]
and the maximum is attained exactly for the eight roots of $x^8 + x^4 + 1$.


\section{Application to the Weierstrass family}

In this section, we consider the family
\[ E_{A,B} \colon y^2 = x^3 + A x + B \]
of elliptic curves. We denote the corresponding division polynomials
by $\Psi_n(A,B,x)$. Then $\Psi_1 = \Psi_2 = 1$ as before, and
\begin{align*}
  \Psi_3 &= 3 x^4 + 6 A x^2 + 12 B x - A^2 \\
         &\equiv -(A - x^2)^2 \bmod 4 \Z[A,B,x] \\
  \Psi_4 &= 2\bigl(x^6 + 5 A x^4 + 20 B x^3 - 5 A^2 x^2 - 4 A B x - (8B^2 + A^3)\bigr) \\
         &\equiv 2 (A - x^2)^3 \bmod 4 \Z[A,B,x] \,.
\end{align*}
We have the same recurrence relations as before, with the factor $4 x (x-1) (x-\lambda)$
replaced by $4 (x^3 + A x + B)$. We apply Proposition~\ref{P:rec},
taking $R = \Z[A,B,x]$, $f = 4 (x^3 + A x + B)^2$ and $g = A - x^2$, to obtain
the following.

\begin{proposition} \label{P:W}
  For all $n \ge 1$, we have
  \[ \Psi_n(A,B,x) \equiv 2^{e(n)} (A - x^2)^{d(n)} \bmod 2^{e(n)+1} \Z[A,B,x] \,. \]
  We also have $\deg_A \Psi_n = d(n)$.
\end{proposition}

For $\alpha \in \C$, let $P_\alpha(A,B)$ (for $4 A^3 + 27 B^2 \neq 0$) be
a point with $x$-coordinate~$\alpha$ on~$E_{A,B}$ and define
\[ T_W(\alpha)
   = \{(A,B) \in \C^2 : 4 A^3 + 27 B^2 \neq 0, P_\alpha(A,B) \in (E_{A,B})_\tors\} \,.
\]
For any subset $\{\alpha_1, \ldots, \alpha_n\} \subseteq \C$, we set
\[ T_W(\alpha_1, \ldots, \alpha_n) = T_W(\alpha_1) \cap \ldots \cap T_W(\alpha_n)\,. \]

\begin{corollary} \label{C:W}
  Let $\alpha, \beta, \gamma \in Z$ such that $\rho(\alpha)$,
  $\rho(\beta)$ and~$\rho(\gamma)$ are pairwise distinct. Then the intersection
  $T_W(\alpha, \beta, \gamma) \cap (\C \times Z)$ is contained in
  \[ \{\bigl(-(\alpha^2+\alpha\beta+\beta^2), \alpha\beta(\alpha+\beta)\bigr),
                 \bigl(-(\alpha^2+\alpha\gamma+\gamma^2), \alpha\gamma(\alpha+\gamma)\bigr),
                 \bigl(-(\beta^2+\beta\gamma+\gamma^2), \beta\gamma(\beta+\gamma)\bigr)\} \,.
  \]
\end{corollary}

\begin{proof}
  Assume that $(A,B) \in T_W(\alpha, \beta, \gamma)$ with $B \in Z$.
  Assume further that at least two of the points $P_\alpha(A,B)$, $P_\beta(A,B)$
  and~$P_\gamma(A,B)$ have order~$\ge 3$, say the first two.
  Then Proposition~\ref{P:W} implies
  that $A \in Z$ and that $\rho(\alpha^2) = \rho(A) = \rho(\beta^2)$,
  which contradicts the assumption. It follows that at least two of the points
  must have order~$2$, say again the first two. We must then have
  \[ \alpha^3 + A \alpha + B = \beta^3 + A \beta + B = 0 \,. \]
  The unique solution of this system of linear equations is
  \[ (A, B) = \bigl(-(\alpha^2+\alpha\beta+\beta^2), \alpha\beta(\alpha+\beta)\bigr) \,. \]
  The other two choices of two points give rise to the other two possible pairs.
\end{proof}

If one could rule out the possibility that $B \notin Z$,
then it would follow that $T_W(\alpha, \beta, \gamma)$ is finite.

What one can say is the following. Assume that $B \notin Z$
and that $P_\alpha(A,B)$ has order~$2$. Then $A = -B/\alpha - \alpha^2$, so $v(A) \le v(B)$.
The polynomials $\Psi_n(A,B,x)$ are weighted-homogeneous of degree~$2 d(n)$
if $x$ has weight~$1$, $A$ has weight~$2$ and $B$ has weight~$3$.
Also, as a polynomial in~$A$, $2^{-e(n)} \Psi_n$ has degree~$d(n)$ and odd
leading coefficient. This implies that in $\Psi_n(A,B,\beta)$ (say),
the term involving the monomial $A^{d(n)}$ will be the unique term with
minimal valuation, hence $\Psi_n(A,B,\beta) \neq 0$. So it remains to exclude
the possibility that all three points have finite order~$\ge 3$ and $v(B) < 0$.

We note that Mavraki~\cite{MavrakiPreprint} studies the case $A = 0$.


\section{The case of transcendence degree 1 in the Legendre family} \label{S:trdeg1}

We now return to the Legendre family.
We have seen above that $T(\alpha,\beta) = \emptyset$ if $\alpha$ and~$\beta$
are algebraically independent over~$\Q$. What can we say when $\Q(\alpha,\beta)$
has transcendence degree~$1$?
Let $F \in \Z[a,b]$ be primitive and irreducible and such that $F(\alpha, \beta) = 0$.
Assume that $\lambda \in T(\alpha,\beta)$. This means that
$\psi_m(\lambda,\alpha) = 0$ for some $m \ge 3$ or $\lambda = \alpha$,
and $\psi_{m'}(\lambda,\beta) = 0$ for some $m' \ge 3$ or $\lambda = \beta$.
We can replace both $m$ and~$m'$ by their least common multiple~$n$.
Eliminating~$\lambda$, we see that $F(a,b)$ must divide the resultant with respect
to~$\lambda$ of $\psi_n(\lambda, a)$ and~$\psi_n(\lambda, b)$, or else $F$ divides
$\psi_n(a,b)$ or~$\psi_n(b,a)$.

\begin{definition}
  For $m \ge 3$, let
  \[ R_m(a,b) = \frac{\Res_{\lambda}\bigl(\psi_m(\lambda,a), \psi_m(\lambda,b)\bigr)}%
                     {(a-b)^{\deg_\lambda \psi_m}}
              \in \Z[a,b] \,.
  \]
\end{definition}

The following result provides the key step in the proof that $T(\alpha, \beta)$
has at most one element in the case of transcendence degree~$1$.

\begin{proposition} \label{P:res}
  For all $m \ge 3$, the polynomial $R_m(a,b)$ is squarefree in~$\Q[a,b]$.
\end{proposition}

\begin{proof}
  We consider the behavior
  of $R_m(a,b)$ as $a$ tends to zero. By Proposition~\ref{P:xfixed}, if
  $\psi_m(\lambda, a) = 0$ and $a \to 0$, then $\lambda \to 0$ as well.
  Since clearly $R_m(a,b)$ divides~$R_n(a,b)$ if $m$ divides~$n$, it is
  sufficient to consider the case that $m = 2n$ is even.

  In the following, we use the symbol $\propto$ to denote equality up to a multiplicative
  constant. By standard properties of resultants, we have
  \[ (b-a)^{n^2-1} R_{2n}(a,b) \propto \prod_{j=1}^{n^2-1} \psi_{2n}(\lambda_j(a), b) \,, \]
  where the $\lambda_j(a)$ are Puiseux series over~$\C$ that represent the roots
  of $\psi_{2n}(\lambda, a)$ as a polynomial in~$\lambda$
  over the power series ring~$\C\pws{a}$.
  Since $\lambda_j$ tends to zero with~$a$, all these series have positive valuation.
  Factoring $\psi_{2n}(\lambda, x) \propto \prod_{j=1}^{2n^2-2} \bigl(x - x_j(\lambda)\bigr)$,
  where $x_j(\lambda)$ are Puiseux series in~$\lambda$, we get the decomposition
  \[ (b-a)^{n^2-1} R_{2n}(a,b)
      \propto \prod_{j=1}^{n^2-1}
        \prod_{j'=1}^{2n^2-2} \bigl(b - (x_{j'} \circ \lambda_j)(a)\bigr) \,.
  \]
  If we can show that the series $x_{j'} \circ \lambda_j$ are pairwise distinct
  (except when $(x_{j'} \circ \lambda_j)(a) = a$, which will occur for a unique~$j'$
  for each~$j$), then this will prove that $R_{2n}(a,b)$ is squarefree.

  To write down these series explicitly, we use the Tate parameterization of~$E_\lambda$.
  Recall that there are power series
  \[ a_4(q) = -5 \sum_{n=1}^\infty \frac{n^3 q^n}{1-q^n} \qquad\text{and}\qquad
    a_6(q) = -\frac{1}{12} \sum_{n=1}^\infty \frac{(7 n^5 + 5 n^3) q^n}{1-q^n}
  \]
  and
  \begin{align*}
    X(u,q) &= \sum_{n=-\infty}^\infty \frac{u q^n}{(1- u q^n)^2}
                - 2 \sum_{n=1}^\infty \frac{q^n}{(1-q^n)^2} \in \Q(u)\pws{q} \\
    Y(u,q) &= \sum_{n=-\infty}^\infty \frac{u^2 q^n}{(1- u q^n)^3}
                + \sum_{n=1}^\infty \frac{q^n}{(1-q^n)^2} \in \Q(u)\pws{q}
  \end{align*}
  such that $\bigl(X(\cdot, q), Y(\cdot, q)\bigr)$ induces a group isomorphism
  of $\C^\times\!/q^{\Z}$ with the $\C$-points on
  $E_{\text{Tate}}(q) \colon y^2 + xy = x^3 + a_4(a) x + a_6(q)$, when $0 < |q| < 1$.
  See for example~\cite{ATAEC}*{Chapter~V}.

  We match this up with~$E_\lambda$: for suitable~$q = Q^2$, we have an isomorphism
  $\phi \colon E_{\text{Tate}}(Q^2) \cong E_\lambda$ such that
  $\phi\bigl(X(-1,Q^2), Y(-1,Q^2)\bigr) = (1,0)$ and
  $\phi\bigl(X(Q,Q^2), Y(Q,Q^2)\bigr) = (0,0)$.
  The $x$-coordinate on~$E_\lambda$ is then given in terms of~$u$ by
  \begin{align*}
    x(u,Q) &= \frac{X(u,Q^2) - X(Q,Q^2)}{X(-1,Q^2) - X(Q,Q^2)} \\
          &= -\frac{4}{(1-u)^2}
                \Bigl(u - 2(1+u)^2 Q
                        + (1+u)^2 (1+8u+u^2) \frac{Q^2}{u} \\
          & \qquad\qquad\qquad\quad{} -8 (1+u)^2 (1+3u+u^2) \frac{Q^3}{u} + \ldots \Bigr) \\
          &\in \Q(u)\pws{Q^2/u} + Q \Q(u)\pws{Q^2/u}
  \end{align*}
  and from $x(-Q,Q) = \lambda$ we have the relation
  \[ \lambda = 16 (Q - 8 Q^2 + 44 Q^3 - 192 Q^4 + 718 Q^5 - 2400 Q^6 + 7352 Q^7 + \ldots) \,. \]
  We use $Q$ as our parameter instead of~$\lambda$ and $\xi(u,Q) = -x(u,Q)/4$
  instead of~$x$; this simplifies the formulas.

  We first consider the series in~$Q$ expressing the $\xi$-coordinates.
  To obtain a further simplification, we set $\xi = \Xi/(1-\Xi)^2$ (with $\Xi$ tending
  to zero with~$\xi$). Then we get the somewhat simpler relation
  \[ 
    \Xi(u, Q) = u - 2 (1-u^2) Q + (1-u^2) (1-3u^2) \frac{Q^2}{u} + 4 (1-u^2)^2 Q^3 + O(Q^4/u) \,.
  \] 
  Fix an $n$th root~$w$ of~$Q$. We set $\zeta_m = \exp(2\pi i/m)$.
  The $\Xi$-coordinates of the $2n$-torsion points are then given
  by~$\Xi(\zeta_{2n}^k w^{\ell}, Q)$,
  where $\ell \in \{0,1,\ldots,n\}$ and $k \in \{0,1,\ldots,2n-1\}$.
  For $\ell = 0$ or $\ell = n$, we restrict to $0 < k < n$
  (this also excludes the 2-torsion points).
  Plugging $u = \zeta_{2n}^k w^{\ell} = \zeta_{2n}^k Q^{\ell/n}$
  into the series for~$\xi$, we obtain the relation
  \[ \Xi_{k,\ell}(Q)
      = \zeta_{2n}^k Q^{\ell/n} - 2 Q + \zeta_{2n}^{-k} Q^{2-\ell/n} + O(Q^{1+2\ell/n})  \,.
  \]
  We set $\gamma_k = \zeta_{2n}^k - 2 + \zeta_{2n}^{-k} = 2(\cos\frac{k\pi}{n} -1)$.
  For $\ell = 0$, we get
  \[ \xi_{k,0}(Q) = \frac{1}{\gamma_k}\bigl(1 - 2(\gamma_k+4)Q + (\gamma_k+4)(\gamma_k+10)Q^2
                                             + \ldots\bigr) \,, 
  \]
  which tends to the nonzero value $\gamma_k^{-1}$ as $Q \to 0$.
  For $0 < \ell < n$, the first two leading terms in~$\Xi_{k,\ell}(Q)$ are
  \[ \Xi_{k,\ell}(Q) = \zeta_{2n}^k Q^{\ell/n} - 2 Q + \ldots \,, \]
  and for $\ell = n$, we have
  \[ \Xi_{k,n}(Q) = \gamma_k Q + 2 \gamma_k (\gamma_k + 2) Q^3 + \ldots \,. \]

  Now we express $Q$ in terms of~$\Xi$. We know that $Q$ tends to zero with~$\Xi$,
  so we must have $0 < \ell \le n$ in the relations above. Solving for $Q$,
  we obtain for $0 < \ell < n$
  \[ Q_{k,\ell}(\Xi) = \zeta_{2\ell}^{-k} \Xi^{n/\ell}
                        + \frac{2n}{\ell} \zeta_{2\ell}^{-2k} \Xi^{2n/\ell-1} + \ldots \,,
  \]
  where we can restrict to $0 \le k < 2\ell$. For $\ell = n$, we get
  \[ Q_{k,n}(\Xi) = \frac{1}{\gamma_k} \Xi
                      - 2 \frac{\gamma_k+2}{\gamma_k^3} \Xi^3 + \ldots \,.
  \]
  Here, $0 < k < n$ as before. In total, we obtain
  \[ \bigl(2 + 4 + 6 + \ldots + (2n-2)\bigr) + (n-1)
      = (n-1)n + (n-1) = n^2 - 1 = d(2n) = \deg_\lambda \psi_{2n}
  \]
  values of~$Q$ in terms of~$\Xi$; this accounts for all possibilities.
  We observe that the $n^2 - 1$ series $Q_{k,\ell}$ all have distinct leading terms
  (note that $0 > \gamma_1 > \gamma_2 > \ldots > \gamma_{n-1} > -4$).

  We first consider the series of the form $\xi_{k',0} \circ Q_{k,\ell}$.
  They have the form
  \[ (\xi_{k',0} \circ Q_{k,\ell})(\alpha)
      = \frac{1}{\gamma_{k'}}\bigl(1 - 2(\gamma_{k'}+4)Q_{k,\ell}(\alpha) + \ldots\bigr) \,.
  \]
  The constant term determines $k'$, and the next term determines the leading
  term of~$Q_{k,\ell}$ and therefore $k$ and~$\ell$. So all these series are
  pairwise distinct (and also distinct from all series $\xi_{k',\ell'} \circ Q_{k,\ell}$
  with $\ell' > 0$, since these series have positive valuation).

  For the remaining series, we work with $\Xi$ instead of~$\xi$, so we consider
  $\Xi_{k',\ell'} \circ Q_{k,\ell}$, where now $\ell' > 0$.
  We obtain the following different cases, where $\ell, \ell' < n$.
  \begin{align*}
    (\Xi_{k',\ell'} \circ Q_{k,\ell})(\alpha)
      &= \zeta_{2\ell n}^{k'\ell-k\ell'} \alpha^{\ell'/\ell}
          + \frac{2\ell'}{\ell} \zeta_{2\ell n}^{k'\ell-k\ell'} \zeta_{2\ell}^{-k}
              \alpha^{(n+\ell')/\ell-1} + \ldots
      & & \text{if $\ell' < \ell$} \\
    (\Xi_{k',\ell'} \circ Q_{k,\ell})(\alpha)
      &= \zeta_{2n}^{k'-k} \alpha
          + 2 (\zeta_{2n}^{k'-k} - 1) \zeta_{2\ell}^{-k} \alpha^{n/\ell} + \ldots
      & & \text{if $\ell' = \ell$} \\
    (\Xi_{k',\ell'} \circ Q_{k,\ell})(\alpha)
      &= \zeta_{2\ell n}^{k'\ell-k\ell'} \alpha^{\ell'/\ell}
          - 2 \zeta_{2\ell}^{-k} \alpha^{n/\ell} + \ldots
      & & \text{if $\ell' > \ell$} \\
    (\Xi_{k',\ell'} \circ Q_{k,n})(\alpha)
      &= \frac{\zeta_{2n}^{k'}}{\gamma_k^{\ell'/n}} \alpha^{\ell'/n}
          - \frac{2}{\gamma_k} \alpha + \ldots
      & & \\
    (\Xi_{k',n} \circ Q_{k,\ell})(\alpha)
      &= \gamma_{k'} \zeta_{2\ell}^{-k} \alpha^{n/\ell}
          + \frac{2n}{\ell} \gamma_{k'} \zeta_{2\ell}^{-2k} \alpha^{2n/\ell-1} + \ldots
      & & \\
    (\Xi_{k',n} \circ Q_{k,n})(\alpha)
      &= \frac{\gamma_{k'}}{\gamma_k} \alpha
          + 2\frac{\gamma_{k'} (\gamma_{k'} - \gamma_k)}{\gamma_{k}^3} \alpha^3 + \ldots
      & &
  \end{align*}
  We note that the second term vanishes if and only if $(k',\ell') = (k,\ell)$,
  in which case we obtain the excluded trivial series~$\alpha$.
  The last case in the list above is distinguished from the second by the fact
  that the leading coefficient has absolute value $\neq 1$. Taking this into
  account, the orders of the first two terms determine $\ell$ and~$\ell'$.
  In all cases but the last, one easily sees that the coefficients of the
  first and the second term together determine $k$ and~$k'$. In the last case, writing
  $\rho = \gamma_{k'}/\gamma_k \neq 1$ for the first coefficient, the second coefficient
  can be written as $2\rho(\rho-1)/\gamma_k$, so both together determine~$k$ and
  then also~$k'$ again.

  So in all cases, the series $\Xi_{k',\ell'} \circ Q_{k,\ell}$ determines
  the two pairs $(k,\ell)$ and~$(k',\ell')$ uniquely (unless $(k,\ell) = (k',\ell')$).
  As noted earlier, this implies the claim.
\end{proof}

Before we deduce consequences of this result, we need to introduce some
further objects. For $n \ge 3$, let $Z_n \subseteq \PP^1_a \times \PP^1_b \times \PP^1_\lambda$
be the curve given by the equations $\psi_n(\lambda, a) = \psi_n(\lambda, b) = 0$,
but excluding the components contained in the plane $a = b$.
Since (for given~$\lambda$) the roots of~$\psi_n(\lambda, x)$ correspond to the
$x$-coordinates of the points in $E_\lambda[n] \setminus E_\lambda[2]$,
the Galois group~$G_n$ of $\psi_n(\lambda, x)$ over~$\Q(\lambda)$ is $\PGL(2, \Z/n\Z)$
when $n$ is odd, and is the subgroup
of $\PGL(2, \Z/n\Z)$ consisting of elements represented by matrices reducing
to the identity mod~$2$ when $n$ is even. Over~$\C(\lambda)$, we have to replace
$\PGL$ by~$\PSL$; write $G'_n$ for the resulting group.

Denote by $T_n$ the set of pairs of
opposite elements of $(\Z/n\Z)^2$ that are not killed by~$2$.
Then the action of~$G_n$ on the roots is the standard action on~$T_n$.
It follows that over~$\C$, $Z_n \to \PP^1_\lambda$ is a Galois covering with group
$G'_n$ acting diagonally on $T_n \times T_n \setminus \Delta$,
where $\Delta$ denotes the diagonal.
Therefore $Z_n$ splits into geometric components corresponding to the orbits
of~$G'_n$ on~$T_n \times T_n \setminus \Delta$.
(The irreducible components over~$\Q$ correspond to the orbits of~$G_n$).

Note that the equation $R_n(a,b) = 0$ describes the projection of~$Z_n$
to $\PP^1_a \times \PP^1_b$. Proposition~\ref{P:res} then says that this
projection maps $Z_n$ birationally onto its image, which we denote~$C_n$.

We can write
\[ \psi_n(\lambda, x) = \prod_{d \mid n} \tilde{\psi}_d(\lambda, x) \,, \]
where $\tilde{\psi}_n(\lambda, x)$, considered as a polynomial in~$x$
over~$\Q(\lambda)$, has as its roots exactly the $x$-coordinates of points
of exact order~$n$ on~$E_\lambda$ (if $n > 2$; we obviously have
$\tilde{\psi}_1 = \tilde{\psi}_2 = 1$).
In~\cite{MasserZannier2013}*{Lemma~2.1}, Masser and Zannier prove that $\tilde{\psi}_n$
is absolutely irreducible if $n \ge 3$ is odd and that $\tilde{\psi}_n$
splits into three irreducible factors in~$\Q[\lambda, x]$, which are
absolutely irreducible if $n \ge 4$ is even (they correspond to fixing
the point of order~$2$ obtained as $(n/2) \cdot P_x(\lambda)$).
We will reserve the term \emph{bicyclotomic polynomial} for these
(absolutely) irreducible factors. So in the notation of~\cite{MasserZannier2013},
a bicyclotomic polynomial~$B(\lambda, x)$ has the form
(note the reversal of the order of the variables)
\[ B(\lambda, x) = B_n^*(x, \lambda) \qquad \text{for $n \ge 3$ odd} \]
or
\[ B(\lambda, x) = B_n^{(0)}(x, \lambda), B_n^{(1)}(x, \lambda) \quad\text{or}\quad
                   B_n^{(\infty)}(x, \lambda) \qquad \text{for $n \ge 4$ even.}
\]
The index~$n$ is the \emph{order} of~$B$.
(There are also the three polynomials $x$, $x-1$ and~$x - \lambda$ of order~$2$,
which we will not call `bicyclotomic'.)

\begin{lemma} \label{L:eqdeg}
  Let $C$ be some geometric irreducible component of~$C_n$, for some~$n \ge 3$.
  There are bicyclotomic polynomials $B_1(\lambda, x)$ and~$B_2(\lambda, x)$
  such that $C$ is contained in the projection of $B_1(\lambda, a) = B_2(\lambda, b) = 0$.
  Let $F(a,b) = 0$ be an equation for~$C$. Then $\deg_a F = \deg_b F$, and
  this degree is a multiple of $\lcm(\deg_\lambda B_1, \deg_\lambda B_2)$.
\end{lemma}

\begin{proof}
  Since each $\psi_n$ is a product of (absolutely irreducible) bicyclotomic
  polynomials, it is clear that every component of~$Z_n$ must be contained
  in a curve of the form $B_1(\lambda, a) = B_2(\lambda, b) = 0$.
  Let $Z$ be the component of~$Z_n$ projecting to~$C$. By Proposition~\ref{P:res},
  the map $\pi \colon Z \to C$ is birational.
  We have the following commuting diagram.
  
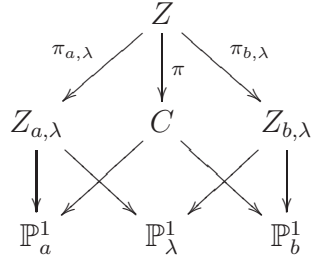
\[ \xymatrix{ & Z \ar[dl]_{\pi_{a,\lambda}} \ar[d]^{\pi} \ar[dr]^{\pi_{b,\lambda}} \\
                Z_{a,\lambda} \ar[d] \ar[dr]
                 & C \ar[dl] \ar[dr]
                 & Z_{b,\lambda} \ar[d] \ar[dl] \\
                \PP^1_a & \PP^1_\lambda & \PP^1_b
              }
  \]
  Here $Z_{a,\lambda} \subseteq \PP^1_a \times \PP^1_\lambda$ is given by $B_1(\lambda, a) = 0$;
  similarly for $Z_{b,\lambda}$.
  Note that $\pi_{a,\lambda}$ is dominant, since $B_1$ is irreducible; similarly
  for~$\pi_{b,\lambda}$. It follows that
  \begin{equation} \label{degrel}
     \deg_b F = (\deg \pi_{a,\lambda}) (\deg_\lambda B_1) \qquad\text{and}\qquad
     \deg_a F = (\deg \pi_{b,\lambda}) (\deg_\lambda B_2) \,.
  \end{equation}
  Considering the two factorizations of $Z \to \PP^1_\lambda$, we obtain
  (using that $\deg_x B_j = 2 \deg_\lambda B_j$)
  \[ 2 \deg_b F = (\deg_x B_1) (\deg \pi_{a,\lambda})
                = (\deg_x B_2) (\deg \pi_{b,\lambda})
                = 2 \deg_a F \,.
  \]
  This shows the equality of degrees, and the relations~\eqref{degrel}
  imply that the common degree is divisible both by~$\deg_\lambda B_1$
  and by~$\deg_\lambda B_2$.
\end{proof}

\begin{corollary}
  No geometric component of any
  of the curves~$C_n$ for $n \ge 3$ satisfies an equation $B(a,b) = 0$ or $B(b,a) = 0$,
  where $B$ is any bicyclotomic polynomial.
\end{corollary}

\begin{proof}
  By Lemma~\ref{L:eqdeg}, the polynomial~$F$ defining a component of~$C_n$
  satisfies $\deg_a F = \deg_b F$. But we have $\deg_x B = 2 \deg_\lambda B$,
  so $F$ cannot be a scalar multiple of~$B$.
\end{proof}

We write $\calC$ for the union of all the curves~$C_n$, together with all
curves given by equations of the form $B(a,b) = 0$ or $B(b,a) = 0$ with
a bicyclotomic polynomial~$B$. The results shown so far imply that
for each (geometric) component~$C$ of~$\calC$ that is not of the form
$B(a,b) = 0$ or $B(b,a) = 0$ for a bicyclotomic polynomial~$B$,
there is a unique $n \ge 3$ such that $C \subseteq C_m$ exactly when $n \mid m$.

\begin{proposition} \label{P:general}
  Let $\alpha, \beta \in \C \setminus \{0,1\}$ with $\alpha \neq \beta$.
  \begin{enumerate}[\upshape (1)]\addtolength{\itemsep}{1mm}
    \item If $(\alpha,\beta) \notin \calC$, then $T(\alpha, \beta) = \emptyset$.
          (This is true whenever $\Q(\alpha,\beta)$ has transcendence degree~$2$.)
    \item If $(\alpha,\beta)$ is a smooth point on $\calC$
          (i.e., it is a smooth point on one component of~$\calC$ and not contained in any
          other component), then $\#T(\alpha, \beta) \le 1$. \\
          In particular, $\#T(\alpha, \beta) \le 1$ whenever $\Q(\alpha,\beta)$
          has transcendence degree~$1$.
    \item If $\#T(\alpha, \beta) \ge 2$, then $(\alpha, \beta)$ is one of the
          countably many singular points of components of~$\calC$ or intersection points
          two distinct components of~$\calC$. In particular, $\alpha$ and~$\beta$
          are algebraic.
  \end{enumerate}
  In general, $\#T(\alpha, \beta)$ is at most the number of branches of $\calC$ passing
  through~$(\alpha, \beta)$.
\end{proposition}

\begin{proof}
  Let $\calZ$ be the union of the~$Z_n$, together with the curves defined by
  $\alpha = \lambda, B(\lambda, \beta) = 0$ or
  by $B(\lambda, \alpha) = 0, \beta = \lambda$. Then $\calZ$ is smooth at all
  points $(\alpha,\beta,\lambda)$ with $\alpha, \beta, \lambda \notin \{0,1,\infty\}$.
  (This is because the $x$-coordinates of torsion points on~$E_\lambda$ are all
  distinct, as long as $\lambda \neq 0, 1, \infty$. In particular, for every~$n$
  the projection of $\{\psi_n(\lambda, x) = 0\} \subseteq \PP^1_x \times \PP^1_\lambda$
  to~$\PP^1_\lambda$ is \'etale over $\PP^1_\lambda \setminus \{0,1,\infty\}$.
  Since $Z_n$ is contained in the fiber square of this projection, its map
  to~$\PP^1_\lambda$ is also \'etale outside $0, 1, \infty$. Including points
  of order~$2$ and passing to the filtered union of the~$Z_n$, we get the claim.)

  By definition, $T(\alpha,\beta)$ is the projection to~$\PP^1_\lambda$
  of the preimage of~$(\alpha,\beta)$ under the projection $\calZ \to \PP^1_a \times \PP^1_b$,
  excluding $\{0,1,\infty\}$.
  Let $\lambda \in \C \setminus \{0,1\}$ be such that $P = (\alpha,\beta,\lambda) \in \calZ$.
  Since $P$ is smooth on~$\calZ$, there is exactly one branch of~$\calC$
  passing through~$(\alpha, \beta)$ that locally is the image of a neighborhood
  of~$P$ in the component of~$\calZ$ it lies on. The results shown above
  imply that no two such branches can coincide. So we get the last statement
  (`In general, \dots') of the proposition; the others follow as special cases.
\end{proof}

Note that the inequality in the second statement of the proposition above
is an equality unless the corresponding value of~$\lambda$ is in~$\{0,1,\infty\}$.
This will never be the case when $\alpha$ or~$\beta$ are transcendental.
So in the case that the transcendence degree of~$\Q(\alpha,\beta)$ is~$1$,
we have the following (where we also use that non-smooth points on~$\calC$
must be algebraic).

\begin{corollary} \label{C:tr1}
  Let $\alpha, \beta \in \C \setminus \{0,1\}$ with $\alpha \neq \beta$
  and such that $\Q(\alpha, \beta)$ has transcendence degree~$1$.
  \begin{enumerate}[\upshape (1)]\addtolength{\itemsep}{1mm}
    \item If $(\alpha,\beta) \in \calC$, then $\#T(\alpha, \beta) = 1$.
    \item Otherwise, $T(\alpha, \beta) = \emptyset$.
  \end{enumerate}
\end{corollary}

Compare this to the upper bound $\#T(\alpha,\beta) \le 6 (12 d)^{32}$
(where $d$ is the degree of an irreducible polynomial $F \in \Q[u,v]$
such that $F(\alpha,\beta) = 0$) given in~\cite{MasserZannier2013}!

We can also improve on~\cite{MasserZannier2013} regarding an effective statement
in this case. If $\alpha$ and~$\beta$ satisfy a polynomial of degree~$d$,
then Masser and Zannier give a bound of~$\pi (12d)^{17/2}$ in the main part
of the paper and of~$180 \pi d \log(180 \pi d)$ in the appendix for the orders
of the corresponding torsion points on~$E_\lambda$ for $\lambda \in T(\alpha,\beta)$.
We observe that the $\lambda$-degree of a bicyclotomic polynomial of order~$n$ is
\[ \delta(n) = \frac{n^2}{4} \prod_{p \mid n} \Bigl(1 - \frac{1}{p^2}\Bigr)
             > \frac{2 n^2}{\pi^2}
\]
when $n$ is odd, and is
\[ \delta(n) = \frac{n^2}{12} \prod_{p \mid n} \Bigl(1 - \frac{1}{p^2}\Bigr)
             > \frac{n^2}{2 \pi^2}
\]
when $n$ is even. By Lemma~\ref{L:eqdeg}, any component of~$\calC$ that is
related to points of order~$n$ must have degree (with respect to $a$ or~$b$)
at least that large. So if $\alpha$ and~$\beta$ are related by an equation
of degree~$d$, this implies that $n < \pi \sqrt{2 d}$. This makes it fairly
easy to enumerate all curves of small degree that are components of~$\calC$.

In fact, we can obtain a list of the degrees of the components of~$\calC$
arising from two given bicyclotomic polynomials by the following combinatorial
approach. Let $\hat{G}_0$ denote the principal congruence subgroup of level~$2$
of~$\GL(2,\hat{\Z})$, where $\hat{\Z}$ is the pro-finite completion of~$\Z$, so
\[ \hat{G}_0 = \bigl\{\gamma \in \GL(2,\hat{\Z})
                      : \gamma \equiv \begin{sm} 1 & 0 \\ 0 & 1 \end{sm} \bmod 2\bigr\} \,,
\]
and let $\hat{G}$ be $\hat{G}_0/\{\pm I\}$, where $I$ is the identity matrix.
Then $\hat{G}$ acts on
\[ M = \bigl((\Q/\Z)^2 \setminus \{0\}\bigr)/\{\pm1\} \,, \]
which, after fixing a basis of the $\Q/\Z$-module~$E_{\lambda,\tors}$,
can be identified with the set of $x$-coordinates of torsion points of order~$\ge 2$
on~$E_\lambda$. In particular, the $\hat{G}$-orbits on~$M$ (except the three
consisting of a point of order~$2$) correspond bijectively
to the bicyclotomic polynomials. Also, $\hat{G}$, with its diagonal action on $M \times M$,
is the automorphism group of the pro-covering $\calZ \to \PP^1_\lambda$ over~$\Q$.
The components of~$\calZ$ (and therefore also the components of its birational
image~$\calC$) correspond bijectively to the orbits of~$\hat{G}$ on~$M \times M$.
Let $O$ be such an orbit, corresponding to the component~$Z$ of~$\calZ$.
Then the projection of~$O \subseteq M \times M$ to the
first factor will be an orbit of~$\hat{G}$ on~$M$, so corresponds to a
bicyclotomic polynomial~$B_1$. Similarly, the projection of~$O$ to the second factor
corresponds to a bicyclotomic polynomial~$B_2$, and $Z$ is contained in the
curve given by $B_1(\lambda,a) = B_2(\lambda,b) = 0$. We assume that none of
the projections consists of a point of order~$2$ (they lead to components of~$\calC$
given by equations $B(a,b) = 0$ or $B(b,a) = 0$, where $B$ is a bicyclotomic polynomial;
these components are easy to describe). Then the component~$C$ of~$\calC$ that
is the projection of~$Z$ is given by an equation $F(a,b) = 0$ with
$\deg_a F = \deg_b F = d$, say. By the considerations
in the proof of Lemma~\ref{L:eqdeg}, we have $d = (\deg \pi_{a,\lambda})(\deg_\lambda B_1)$.
So to determine~$d$, we have to find the degree of the covering $Z \to Z_{a,\lambda}$,
where $Z_{a,\lambda}$ is given by~$B_1(\lambda,a) = 0$. But fixing a point
on~$Z_{a,\lambda}$ corresponds to fixing a representative $m \in M$ of the
projection of~$O$ to the first component. Up to changing the basis of~$E_{\lambda,\tors}$
used for the identification with~$(\Q/\Z)^2$, we can take $m = \tfrac{1}{n} \bmod \Z$,
where $n$ is the order of the points whose $x$-coordinates are the roots
of~$B_1(\lambda,{\cdot})$. Then $\deg \pi_{a,\lambda}$ is the size of the fiber
of~$O$ above~$m$. The possible fibers are the orbits of the stabilizer of~$m$
in~$\hat{G}$ on the subset~$M_2$ of~$M$ corresponding to~$B_2$. If the order of the
points coming from~$B_2$ is~$n'$, then the relevant group is
\[ G_{n,n'} = \bigl\{\gamma \in \GL(2,\Z/n'\Z)
                      : \gamma \equiv I \bmod \gcd(n',2),
                        \gamma \equiv \begin{sm} 1 & * \\ 0 & * \end{sm} \bmod \gcd(n,n')
              \bigr\}/\{\pm I\} \,,
\]
acting on~$M_2$. This allows us to find the degrees of all components of~$\calC$
arising from $B_1$ and~$B_2$. To illustrate this, we present a table giving the
number of components of~$\calC$ for small bidegrees (the \emph{bidegree} of
$F \in \Q[a,b]$ is the pair $(\deg_a F, \deg_b F)$).
\[ \renewcommand{\arraystretch}{1.2}
   \begin{array}{|r||*{9}{c|}} \hline
     \text{bidegree}
       & (1,2) & (2,4) & (4,8) & (6,12) & (8,16) & (12,24) & (16,32) & (18,36) & (24,48) \\\hline
     \#\text{components}
       &   3   &   4   &   3   &   4    &   3    &    4    &    3    &    4  & 3 \\\hline\hline
     \text{bidegree}
       & (2,1) & (4,2) & (8,4) & (12,6) & (16,8) & (24,12) & (32,16) & (36,18) & (48,24) \\\hline
     \#\text{components}
       &   3   &   4   &   3   &   4    &   3    &    4    &    3    &    4  & 3 \\\hline\hline
     \text{bidegree}
       & (1,1) & (2,2) & (4,4) & (6,6) & (8,8) & (12,12) & (16,16) & (18,18) & (24,24) \\\hline
     \#\text{components}
       &   3   &  18   &  45   &  44   &  57   &   68    &   96    &   76    &   161 \\\hline
   \end{array}
\]

This is related to the sets $\calF_d$ defined in~\cite{MasserZannier2013},
which consist of polynomials of total degre~$d$ defining components of~$\calC$.
Working with the bidegree instead of the total degree appears to be more natural,
since it is invariant under
the action of the group~$S_3$ (generated by the involutions
$(a,b) \mapsto (1-a,1-b)$ and $(a,b) \mapsto (1/a,1/b)$)
that stabilizes~$\calC$.

The file at~\cite{Allcurves}, when read into~Magma, results in a list of defining
polynomials~$F$ for all components of~$\calC$ such that $\deg_a F + \deg_b F \le 192$.

\begin{table}[htb]
  \hrulefill
  \begin{align*}
  (1,1) \colon \quad& a + b, \quad a + b - 2, \quad 2 a b - a - b. \\
  (1,2) \colon \quad& a - b^2, \quad a + b^2 - 2 b,\quad 2 a b - a - b^2. \\
  (2,1) \colon \quad& a^2 - b, \quad a^2 - 2 a + b, \quad a^2 - 2 a b + b. \\
  (2,2) \colon \quad& a^2 + b^2 - 2 b, \quad a^2 - 2 a + b^2, \quad
                      a^2 - 2 a b^2 + b^2, \quad 2 a^2 b - a^2 - b^2, \\
        & a^2 + 2 a b - 4 a + b^2, \quad a^2 + 2 a b + b^2 - 4 b, \quad
                a^2 - 2 a b - 3 b^2 + 4 b, \\
        & a^2 + 2 a b^2 - 4 a b - b^2 + 2 b, \quad a^2 + 4 a b^2 - 2 a b - 3 b^2, \quad
                a^2 + 4 a b^2 - 6 a b - 3 b^2 + 4 b, \\
        & a^2 - 4 a b^2 + 2 a b + b^2, \quad a^2 - 4 a b^2 + 6 a b - 4 a + b^2, \quad
                3 a^2 + 2 a b - 4 a - b^2, \\
        & 2 a^2 b - a^2 - 4 a b + 2 a + b^2, \quad 4 a^2 b - a^2 - 2 a b - b^2, \quad
                4 a^2 b - a^2 - 6 a b - b^2 + 4 b, \\
        & 4 a^2 b - 3 a^2 - 2 a b + b^2, \quad
                4 a^2 b - 3 a^2 - 6 a b + 4 a + b^2. \\
  (2,4) \colon \quad& a^2 + 4 a b^3 - 6 a b^2 - 3 b^4 + 4 b^3, \quad
                      a^2 - 4 a b^3 + 6 a b^2 - 4 a b + b^4, \\
        & 4 a^2 b - a^2 - 6 a b^2 - b^4 + 4 b^3, \quad 4 a^2 b - 3 a^2 - 6 a b^2 + 4 a b + b^4. \\
  (4,2) \colon \quad& a^4 - 4 a^3 b + 6 a^2 b - 4 a b + b^2, \quad
                      a^4 - 4 a^3 + 6 a^2 b - 4 a b^2 + b^2, \\
        & a^4 - 6 a^2 b + 4 a b^2 + 4 a b - 3 b^2, \quad 3 a^4 - 4 a^3 b - 4 a^3 + 6 a^2 b - b^2.
  \end{align*}
  \caption{Polynomials of small bidegree defining components of~$\calC$} \label{Tablebi}

  \hrulefill
\end{table}

As an illustration, in Table~\ref{Tablebi}
we list the 35~polynomials of bidegrees $(1,1)$, $(1,2)$, $(2,1)$, $(2,2)$, $(2,4)$ and~$(4,2)$
defining components of~$\calC$.
The correctness of the list for bidegree~$(1,1)$ provides a simple proof
of~\cite{MasserZannier2013}*{Theorem~2}. In general, we obtain the following refinement
of Corollary~\ref{C:tr1}.

\begin{corollary}
  Let $\alpha, \beta \in \C \setminus \{0,1\}$ with $\alpha \neq \beta$ such that
  $\Q(\alpha, \beta)$ has transcendence degree~$1$. Then $\#T(\alpha, \beta) \le 1$.
  If we are given an irreducible Polynomial $F \in \Q[a,b]$ such that $F(\alpha, \beta) = 0$,
  then we can effectively determine the set~$T(\alpha, \beta)$.
\end{corollary}

Note that for the effectivity statement, we need to \emph{know} that $\alpha$ and~$\beta$
generate a field of transcendence degree~$1$, and we need to \emph{know} the algebraic
relation linking them. For example, we cannot say whether $T(e, \pi)$ is empty or not,
since we do not know whether $e$ and~$\pi$ are algebraically dependent or not.

We have computed the complete list of all polynomials of bidegree~$(d,d)$
defining components of~$\calC$ for all~$d$ up to~$96$.
They were obtained either by computing resultants like~$R_m$
(but using two bicyclotomic polynomials instead of twice~$\psi_m$) and
factoring the result, or by using the relation between $x(P)$ and~$x(nP)$ for suitable~$n$,
in cases where such a dependency was satisfied in the relevant $\hat{G}$-orbit on~$M \times M$.
We used these polynomials as input for the computations described in the next section.


\section{Speculations on the size of $T(\alpha,\beta)$} \label{S:size}

In the last case of Proposition~\ref{P:general} one would in general not expect that
more than two branches pass through the same point: such intersections
are unlikely. This leads to the following (perhaps somewhat bold) conjecture.

\begin{conjecture} \label{conj2}
  There are only finitely many pairs $(\alpha, \beta) \in \bar{\Q} \times \bar{\Q}$
  with $\alpha, \beta \notin \{0,1\}$, $\alpha \neq \beta$
  and $\#T(\alpha,\beta) \ge 3$.
\end{conjecture}

It would be interesting to investigate if this conjecture would follow
from some version of the Zilber-Pink Conjecture(s).

To get some evidence related to this question, we took all irreducible components over~$\Q$
of~$\calC$ given by equations $f(a,b) = 0$ with
$\deg_{ab} f \mathrel{:=} \deg_a f + \deg_b f \le 192$ (as mentioned at the end
of the previous section, we had computed all these equations).
We then computed all intersections of two components $\{f_1 = 0\}$ and $\{f_2 = 0\}$
such that $(\deg_{ab} f_1) \cdot (\deg_{ab} f_2) \le 384$, and all singularities
of components $\{f = 0\}$ with $(\deg_{ab} f)^2 \le 384$. After removing parts contained
in $\{a(1-a)b(1-b)(a-b) = 0\}$, we split the resulting finite schemes into irreducible
components over~$\Q$ and computed the set $T_{50}(\alpha, \beta)$ for a representative
point $(\alpha, \beta)$ for each of these irreducible components. Here
\[ T_n(\alpha, \beta) = \{\lambda \in \C \setminus \{0,1\} :
                          \text{$P_\alpha(\lambda)$ and $P_\beta(\lambda)$
                          are both torsion of order $\le n$}\} \,.
\]
To reduce the amount of computation, we make use of the fact that the group
$\Gamma = S_3 \times C_2$ acts on~$\calC$, where the action of~$S_3$ is diagonal
on both coordinates and generated by the involutions $x \mapsto 1-x$ and $x \mapsto 1/x$,
and the action of the cyclic group~$C_2$ is given by swapping the coordinates.
(This action lifts to an action on~$\calZ$, where $S_3$ acts diagonally on $(a,b,\lambda)$.)
This implies that $\Gamma$ also acts on the countable disjoint union of $\Q$-integral
finite schemes making up the set of all $(\alpha, \beta)$ with $\#T(\alpha, \beta) \ge 2$.
It is therefore sufficient to find one representative scheme in each $\Gamma$-orbit.

This computation produced 82\,717~irreducible finite schemes, falling into
8\,083~$\Gamma$-orbits, and containing 2\,212\,784~geometric points in total, consisting
of points $(\alpha, \beta)$ such that $\#T_{50}(\alpha, \beta) \ge 2$. Of these, 180~schemes
making up 24~orbits and containing 558~geometric points have sets $T_{50}(\alpha, \beta)$
with three or more elements. This supports Conjecture~\ref{conj2} in that it shows
that such pairs are quite rare. A list of representatives of all such orbits
with $\#T_{50}(\alpha, \beta) = 2$
is obtained by loading the file at~\cite{Twolambdas} into Magma.
The orbits with $\#T_{50}(\alpha, \beta) \ge 3$ are given with more detailed
information in the file at~\cite{Atleast3}.
The latter orbits are as follows. We begin with those that have $\#T_{50} = 3$.

There is one example with $\alpha, \beta \in \Q$, which is
represented by
\[ \{\alpha, \beta\} = \Bigl\{\frac{3}{8}, -\frac{9}{16}\Bigr\}
   \qquad\text{with}\qquad
   T_{50}(\alpha,\beta) = \Bigl\{-\frac{9}{16}, \frac{3}{128}, \frac{81}{256}\Bigr\} \,.
\]
Note that $\alpha$ and~$\beta$
both reduce to~$\infty$ mod~$2$, illustrating Corollary~\ref{cor1}.
The orders of the points with $x$-coordinate $\alpha$ and~$\beta$ on~$E_\lambda$
are $(4,2)$, $(6,6)$ and~$(8,4)$.

Then there are six examples with $\alpha$ and~$\beta$ in a quadratic field, as listed
in the following table.
The entry `orders' records the torsion orders of the points with $x$-coordinate
$\alpha$ and~$\beta$, for each of the given values of~$\lambda$.
(Note that this order may change by a factor of~$2$ under the $S_3$-action,
since $x \mapsto 1/x$ interchanges the origin on~$E_\lambda$ with a point of
order~$2$.)
{\renewcommand{\arraystretch}{1.5}
\[ \begin{array}{|c|c||c||c|}
      \alpha & \beta & T_{50}(\alpha, \beta) & \text{orders} \\\hline
      \frac{7+5\sqrt{-7}}{14} & \frac{7-11\sqrt{-7}}{14} &
       \{\frac{7-11\sqrt{-7}}{14}, \frac{21+31\sqrt{-7}}{42}, \frac{49-13\sqrt{-7}}{98}\} &
       (6,2), (8,8), (6,6) \\
      1+\sqrt{2} & -1+\sqrt{2} & \{-1, 7-4\sqrt{2} \pm (4-2\sqrt{2})\sqrt{2-\sqrt{2}}\} &
       (4,4), (5,10), (5,10) \\
      \frac{5+3\sqrt{-15}}{10} & \frac{15-7\sqrt{-15}}{30} &
       \{\frac{5-13\sqrt{-15}}{10}, \frac{45+11\sqrt{-15}}{90},
         \frac{75+61\sqrt{-15}}{150}\} &
         (6,3), (4,6), (12,6) \\
      \frac{15+7\sqrt{-15}}{30} & \frac{45-11\sqrt{-15}}{90} &
       \{\frac{27+19\sqrt{-15}}{54}, \frac{45-11\sqrt{-15}}{90},
         \frac{75-61\sqrt{-15}}{150}\} &
         (5,5), (6,2), (6,3) \\
      \frac{1+\sqrt{17}}{2} & \frac{-7+\sqrt{17}}{2} &
       \{\frac{-7+\sqrt{17}}{2}, \frac{33-7\sqrt{17}}{2}, \frac{-31-7\sqrt{17}}{2}\} &
       (4,2), (3,4), (3,6) \\
      \frac{-7+\sqrt{17}}{2} & \frac{33-7\sqrt{17}}{2} &
       \{\frac{33-7\sqrt{17}}{2}, \frac{-895+217\sqrt{17}}{2}, \frac{3+11\sqrt{17}}{6}\} &
       (4,2), (6,4), (10,10) \\
      \hline
   \end{array}
\]
}

There are nine examples over quartic fields,
one each over $\Q(\sqrt{2},\sqrt{3})$ (orders ($(4,4)$, $(6,3)$, $(10,5)$),
$\Q(\sqrt{-3},\sqrt{13})$ (orders $(4,6)$, $(6,4)$, $(10,10)$),
a dihedral field containing~$\Q(\sqrt{2})$ (orders $(4,8)$, $(8,4)$, $(10,10)$),
one containing~$\Q(\sqrt{5})$ (orders $(4,4)$, $(6,7)$, $(6,7)$),
one containing~$\Q(\sqrt{-7})$ (orders $(4,2)$, $(6,6)$, $(7,7)$),
one containing~$\Q(\sqrt{17})$ (orders $(3,4)$, $(4,6)$, $(5,10)$)
and one containing~$\Q(\sqrt{33})$ (orders $(3,6)$, $(6,4)$, $(9,6)$).
Another dihedral field containing~$\Q(\sqrt{17})$ shows up twice, with
orders $(2,6)$, $(4,6)$, $(7,7)$ and $(4,6)$, $(5,5)$, $(7,7)$.

There are two examples over fields of degree six, one over $\Q(\zeta_7)$
with orders $(7,7)$, $(8,8)$, $(9,9)$, and one over a quadratic extension
of the cubic field of discriminant~$-31$ (orders $(3,6)$, $(6,6)$, $(7,7)$).
In addition, there is one example over the octic field~$\Q(\zeta_{15})$
(orders $(8,4)$, $(8,8)$, $(6,10)$), one over another octic field that is
Galois over~$\Q$ and a quadratic extension of~$\Q(\sqrt{-1}, \sqrt{5})$
(orders $(4,4)$, $(7,7)$, $(7,7)$),
and one example over a field of degree~$16$ (orders $(6,6)$, $(10,10)$, $(10,10)$).

Then there are two further orbits that have $\#T_{50} = 4$.
One is represented by
\begin{align*}
  \{\alpha, \beta\} &= \Bigl\{\frac{-7 + \sqrt{17}}{2}, \frac{9 + \sqrt{17}}{2}\Bigr\}
   \qquad\text{with} \\
   T_{50}(\alpha,\beta) &= \Bigl\{\frac{-31-7\sqrt{17}}{2}, \frac{33-7\sqrt{17}}{2},
                              \frac{17-23\sqrt{17}}{34}, \frac{3+11\sqrt{17}}{6}\Bigr\}\,.
\end{align*}
The pairs of orders of the points are $(6,4), (4,6), (6,6), (10,10)$.
The other example is over the (sextic) Hilbert class field of~$\Q(\sqrt{-23})$
and has orders $(3,6)$, $(6,6)$, $(7,7)$ and~$(14,14)$.

Finally, we have
\begin{align*}
  \{\alpha, \beta\} = \{\sqrt{-1}, -\sqrt{-1}\}
   \qquad\text{with}\qquad
   T_{100}(\alpha,\beta) = \Bigl\{-1, 3 \pm 2\sqrt{2}, \frac{1 \pm 2\sqrt{-2}}{3} \Bigr\}
\end{align*}
of size~$5$.
The pairs of orders of the points are here $(4,4), (6,6), (6,6), (10,10), (10,10)$.
All the number fields occurring here have trivial class group, except for
$\Q(\sqrt{-15})$ and~$\Q(\sqrt{-3}, \sqrt{13})$, which both have class number~$2$.

Note that for any given bicyclotomic polynomial~$B(\lambda, x)$, we can take
the factors of~$B(\lambda, \alpha)$ as a polynomial in~$\lambda$ over~$\Q(\alpha)$
and for each factor~$f$ consider a root~$\lambda_0$ of~$f$ and check if the points
on~$E_{\lambda_0}$ with $x$-coordinate~$\beta$ have finite order or not
(for example by considering the reductions modulo suitable prime ideals of
small degree; note that a point of finite order reduces to a point of the same
finite order modulo all odd primes of good reduction, so if we find two different
orders in this way, we know that the point must have infinite order).
In this way, we checked that for any $\lambda \in T(\sqrt{-1}, -\sqrt{-1})$ not
in the list above, the orders of both pairs of points on~$E_\lambda$ must be
larger than~$200$. We also found that all~$B(\lambda, \sqrt{-1})$ of orders
up to~$200$ are irreducible over~$\Q(\sqrt{-1})$, with one exception at order~$10$.
In addition, we checked for each of the other 23~orbits of points with $\#T_{50} \ge 3$
that for any unknown $\lambda \in T(\alpha, \beta)$ both points $P_\alpha(\lambda)$
and~$P_\beta(\lambda)$ must have order $> 100$.
This suggests the following conjecture.

\begin{conjecture} \label{conj-b5} \strut\vspace{-2mm}
  \begin{enumerate}[\upshape (1)]\addtolength{\itemsep}{1mm}
    \item $T(\sqrt{-1}, -\sqrt{-1}) = \{-1, 3 \pm 2\sqrt{2}, \tfrac{1}{3} \pm \tfrac{2}{3}\sqrt{-2}\}$.
    \item All other sets $T(\alpha, \beta)$ have at most four elements.
  \end{enumerate}
\end{conjecture}

We also remark that in our computations, we never found more than two branches
through any singular point of an irreducible component of~$\calC$. So we propose:

\begin{conjecture} \label{conjsing}
  There is a number~$N$ such that for any irreducible component~$C$ of~$\calC$
  and any (singular) point~$P$ on~$C$ outside the `bad set' given by $a(a-1)b(b-1)(a-b) = 0$,
  there are at most~$N$ branches of~$C$ through~$P$ (equivalently, $P$ has at most~$N$
  preimages in the component~$Z$ of~$\calZ$ that maps birationally to~$C$).
\end{conjecture}

Our computations suggest that perhaps one can even take $N = 2$. Note that it
would follow that every $\lambda \in T(\alpha, \beta)$ has the property that
$[\Q(\alpha, \beta, \lambda) : \Q(\alpha, \beta)] \le N$. This implies that
Conjecture~\ref{conjsing} with an explicit~$N$ would give an \emph{effective}
procedure for determining $T(\alpha, \beta)$ for algebraic $\alpha$ and~$\beta$.
Namely, we can find an explicit bound for the height of the elements of~$T(\alpha)$,
say (see~\cite{Silverman1983});
together with the bound on the degree, this leaves finitely many candidates
for~$\lambda$, which we can check for membership in~$T(\alpha, \beta)$.

\begin{table}[htb]
  \hrulefill
  
  \[ \begin{array}{|r||c|c|c|c|c|c|c|}
      \text{degree} & [0, 1) & [1, 2) & [2, 3) & [3, 4) & [4, 5) & [5, 6) & [6, \infty) \\\hline
                  1 &     0  &     0  &     1  &     3  &     0  &     2  &     0  \\
                  2 &     7  &    13  &    16  &    19  &    12  &     5  &     2  \\
                  3 &     6  &    15  &    37  &    17  &     5  &     1  &     0  \\
                  4 &    21  &    44  &    78  &    34  &    23  &     2  &     4  \\
                  5 &    13  &    37  &    47  &    19  &     3  &     1  &     0  \\
                  6 &    28  &    68  &   125  &    44  &     7  &     7  &     0  \\
                  7 &     0  &    30  &    58  &    20  &     3  &     0  &     0  \\
                  8 &    10  &   104  &   105  &    40  &    11  &     3  &     0  \\
                  9 &     2  &    52  &    64  &    22  &     3  &     0  &     0  \\
                 10 &    13  &    92  &   113  &    48  &     5  &     2  &     0  \\
                 11 &     3  &    51  &    66  &    15  &     1  &     1  &     0  \\
                 12 &    14  &    94  &   137  &    41  &    10  &     7  &     0  \\
                 13 &     1  &    59  &    55  &     9  &     5  &     0  &     0  \\\hline
      \text{total\large\strut}
                    &   118  &   659  &   902  &   331  &    88  &    31  &     6  \\\hline
    \end{array}
  \]
  
  \caption{Distribution of heights of schemes with $\#T \ge 2$} \label{Tableht}
  \hrulefill
\end{table}

Our computations also suggest the following further conjecture.

\begin{conjecture} \label{conj3}
  Fix $d \ge 1$. Then there are only finitely many $\alpha,\beta \in \bar{\Q}$,
  both of degree at most~$d$, with $\alpha, \beta \notin \{0,1\}$, $\alpha \neq \beta$
  and $\#T(\alpha, \beta) \ge 2$.
\end{conjecture}

For example, it appears that other than the orbit mentioned above with
$\#T(\alpha, \beta) \ge 3$, there might be only six further orbits (each of size~$12$)
of pairs of rational numbers~$\alpha$ and~$\beta$ with $\#T(\alpha, \beta) \ge 2$,
represented by $(-3,3)$, $(-5/4, 5/2)$, $(-4/5, 8/5)$, $(-3,9)$, $(-9/16, 9/4)$
and~$(-27/5, -3/5)$. Note that for the first two of these, one can deduce that
in fact $\#T(\alpha, \beta) = 2$ via Corollary~\ref{C:rat}.

Conjecture~\ref{conj3} would follow from the following:

\begin{conjecture} \label{conj4}
  There is an absolute bound $B$ such that for all
  $\alpha, \beta \in \bar{\Q} \setminus \{0,1\}$ with $\alpha \neq \beta$
  and $\#T(\alpha, \beta) \ge 2$, we have $h(\alpha), h(\beta) \le B$.
\end{conjecture}

Here $h \colon \bar{\Q} \to \R_{\ge 0}$ denotes the absolute logarithmic height.

We set $\bar{h}(\alpha) = \bigl(h(\alpha) + h(1-\alpha) + h(1-1/\alpha)\bigr)/3$;
then $\bar{h}$ is invariant under the $S_3$-action on~$\PP^1(\bar{\Q})$. Note
that
\[ \bar{h}(\alpha) - \frac{2}{3} \log 2 \le h(\alpha) \le \bar{h}(\alpha) + \frac{1}{3} \log 2 \]
(with equality on the left for $\alpha = -1$ and on the right for $\alpha \in \{2,\frac{1}{2}\}$),
so that we could formulate an equivalent conjecture using $\bar{h}$ instead of~$h$.
To test Conjecture~\ref{conj4}, we computed $\bar{h}(\alpha) + \bar{h}(\beta)$ for a
representative point $(\alpha, \beta)$ in each $\Gamma$-orbit of points of degree
at most~$13$ that we encountered in our computation. This gave rise to the
statistics in Table~\ref{Tableht}, where we give the distribution of these height sums
according to intervals of length~$1$
for the orbits of given degree (excluding those with $\#T(\alpha, \beta) \ge 3$).

The largest height sum encountered was $\approx 6.723796$, occurring for degree~$4$.
There is no tendency towards a systematic increase of these heights with increasing
degree of the points or with increasing degree of the curves that were intersected.
This lends some credibility to Conjecture~\ref{conj4}.


\end{document}